\documentclass[a4paper,twoside]{article}
\usepackage{a4}
\usepackage{amssymb}
\usepackage{amsmath}
\usepackage{url}
\usepackage[active]{srcltx}
\usepackage[pagebackref,colorlinks,citecolor=blue,linkcolor=blue,urlcolor=blue]{hyperref}
\usepackage[dvipsnames]{color}
\usepackage{upref}
\allowdisplaybreaks[2] 
\newcount\minutes \newcount\hours
\hours=\time
\divide\hours 60
\minutes=\hours
\multiply\minutes -60
\advance\minutes \time
\newcommand{\klockan}{\the\hours:{\ifnum\minutes<10 0\fi}\the\minutes}
\newcommand{\tid}{\today\ \klockan}
\newcommand{\prtid}{\smash{\raise 10mm \hbox{\LaTeX ed \tid}}}
\renewcommand{\prtid}{}
\makeatletter
\def\sectionmark#1{} 
\def\subsectionmark#1{}
\newcommand{\sectnr}{\ifnum \c@secnumdepth >\z@
                 \thesection.\hskip 1em\relax \fi}
\def\@evenhead{\footnotesize\rm\thepage\hfil\leftmark\hfil\llap{\prtid}}
\def\@oddhead{\footnotesize\rm\rlap{\prtid}\hfil\rightmark\hfil\thepage}
\def\tableofcontents{\section*{Contents} 
 \@starttoc{toc}}
\makeatother
\makeatletter
\def\@biblabel#1{#1.}
\makeatother
\makeatletter
\let\Thebibliography=\thebibliography
\renewcommand{\thebibliography}[1]{\def\@mkboth##1##2{}\Thebibliography{#1}
\addcontentsline{toc}{section}{References}
\frenchspacing 
\setlength{\@topsep}{0pt}
\setlength{\itemsep}{0pt}%
\setlength{\parskip}{0pt plus 2pt}%
}
\makeatother
\makeatletter
\def\mdots@{\mathinner.\nonscript\!.%
 \ifx\next,.\else\ifx\next;.\else\ifx\next..\else
 \nonscript\!\mathinner.\fi\fi\fi}
\let\ldots\mdots@
\let\cdots\mdots@
\let\dotso\mdots@
\let\dotsb\mdots@
\let\dotsm\mdots@
\let\dotsc\mdots@
\def\vdots{\vbox{\baselineskip2.8\p@ \lineskiplimit\z@
    \kern6\p@\hbox{.}\hbox{.}\hbox{.}\kern3\p@}}
\def\ddots{\mathinner{\mkern1mu\raise8.6\p@\vbox{\kern7\p@\hbox{.}}%
    \raise5.8\p@\hbox{.}\raise3\p@\hbox{.}\mkern1mu}}
\makeatother
\makeatletter
\let\Enumerate=\enumerate
\renewcommand{\enumerate}{\Enumerate%
\setlength{\@topsep}{0pt}
\setlength{\itemsep}{0pt}%
\setlength{\parskip}{0pt plus 1pt}%
\renewcommand{\theenumi}{\textup{(\alph{enumi})}}%
\renewcommand{\labelenumi}{\theenumi}%
}
\let\endEnumerate=\endenumerate
\renewcommand{\endenumerate}{\endEnumerate\unskip}
\makeatother
\makeatletter
\def\@seccntformat#1{\csname the#1\endcsname.\quad}
\makeatother
\newcommand{\authortitle}[2]{\author{#1}\title{#2}\markboth{#1}{#2}}
\newcommand{\auth}[2]{{#1, #2.}}

\newcommand{\art}[6]{{\sc #1, \rm #2, \it #3 \bf #4 \rm (#5), \mbox{#6}.}}

\newcommand{\arttoappearin}[3]{{\sc #1, \rm #2,  to appear in #3.}}
\newcommand{\artprep}[3]{{\sc #1, \rm #2, #3.}}

\newcommand{\book}[3]{{\sc #1, \it #2, \rm #3.}}
\newcommand{\AND}{{\rm and }}

\newcommand{\arXiv}[1]{{\tt \href{https://arxiv.org/abs/#1}{arXiv:#1}}}
\RequirePackage{amsthm}
\newtheoremstyle{descriptive}%
  {\topsep}   
  {\topsep}   
  {\rmfamily} 
  {}          
  {\bfseries} 
  {.}         
  { }         
  {}          
\newtheoremstyle{propositional}%
  {\topsep}   
  {\topsep}   
  {\itshape}  
  {}          
  {\bfseries} 
  {.}         
  { }         
  {}          
\theoremstyle{propositional}
\newtheorem{thm}{Theorem}[section]
\newtheorem{prop}[thm]{Proposition}
\newtheorem{lem}[thm]{Lemma}
\newtheorem{cor}[thm]{Corollary}
\theoremstyle{descriptive}

\newtheorem{remark}[thm]{Remark}
\makeatletter
\renewenvironment{proof}[1][\proofname]{\par
  \pushQED{\qed}%
  \normalfont
  \trivlist
  \item[\hskip\labelsep
        \itshape
    #1\@addpunct{.}]\ignorespaces
}{%
  \popQED\endtrivlist\@endpefalse
}
\makeatother
\newcommand{\setm}{\setminus}
\renewcommand{\emptyset}{\varnothing}
\DeclareMathOperator{\diam}{diam}
\newcommand{\bdry}{\partial}
\newcommand{\bdy}{\bdry}
\newcommand{\simge}{\gtrsim}
\newcommand{\simle}{\lesssim}
\newcommand{\al}{\alpha}
\newcommand{\alp}{\alpha}
\newcommand{\de}{\delta}
\newcommand{\eps}{\varepsilon}
\newcommand{\La}{\Lambda}
\newcommand{\Om}{\Omega}
\newcommand{\s}{\sigma}
\newcommand{\p}{{$p\mspace{1mu}$}}
\newcommand{\binfty}{{\boldsymbol{\infty}}}
\newcommand{\R}{\mathbf{R}}
\newcommand{\Rn}{{\R^n}}
\newcommand{\limplus}{{\mathchoice{\vcenter{\hbox{$\scriptstyle +$}}}
  {\vcenter{\hbox{$\scriptstyle +$}}}
  {\vcenter{\hbox{$\scriptscriptstyle +$}}}
  {\vcenter{\hbox{$\scriptscriptstyle +$}}}
}}
\newcommand{\ka}{\kappa}
\newcommand{\nuhat}{{\hat{\nu}}}
\newcommand{\nuinv}{{\nu}^{-1}}
\newcommand{\nutilde}{{\tilde{\nu}}}
\newcommand{\Zhat}{\widehat{Z}}
\newcommand{\Bhat}{\widehat{B}}
\newcommand{\dhat}{\hat{d}}
\newcommand{\mhat}{\widehat{m}}
\newcommand{\rhohat}{\hat{\rho}}
\newcommand{\rhotilde}{\tilde{\rho}}
\newcommand{\dtilde}{\tilde{d}}
\newcommand{\bhat}{\hat{b}}
\newcommand{\Zt}{\widetilde{Z}}
\newcommand{\kahat}{\hat{\ka}}

\numberwithin{equation}{section}
\newcommand{\eqv}{\ensuremath{
\mathchoice{\quad \Longleftrightarrow \quad}{\Leftrightarrow}
                {\Leftrightarrow}{\Leftrightarrow}}}
\newcommand{\imp}{\ensuremath{
\mathchoice{\quad \Longrightarrow \quad}{\Rightarrow}
                {\Rightarrow}{\Rightarrow}}}
\newenvironment{ack}{\medskip{\it Acknowledgement.}}{}

\begin{document}

\authortitle{Anders Bj\"orn, Jana Bj\"orn, Riikka Korte, Sari Rogovin and Timo Takala}
{Preserving Besov (fractional Sobolev) energies
 under sphericalization and flattening}

\author{
Anders Bj\"orn \\
\it\small Department of Mathematics, Link\"oping University, SE-581 83 Link\"oping, Sweden\\
\it \small anders.bjorn@liu.se, ORCID\/\textup{:} 0000-0002-9677-8321
\\
\\
Jana Bj\"orn \\
\it\small Department of Mathematics, Link\"oping University, SE-581 83 Link\"oping, Sweden\\
\it \small jana.bjorn@liu.se, ORCID\/\textup{:} 0000-0002-1238-6751
\\
\\
Riikka Korte \\
\it\small Department of Mathematics and Systems Analysis, Aalto University, \\
\it\small PL 11100, FI-00076 Aalto, Finland\\
\it \small riikka.korte@aalto.fi, ORCID\/\textup{:} 0000-0002-6313-2233
\\
\\
Sari Rogovin \\
\it\small Department of Mathematics and Systems Analysis, Aalto University,\\
\it\small  PL 11100, FI-00076 Aalto, Finland\\
\it \small sari.rogovin@aalto.fi, ORCID\/\textup{:} 0009-0001-3713-637X
\\
\\
Timo Takala \\
\it\small Department of Mathematics and Systems Analysis, Aalto University, \\
\it\small PL 11100, FI-00076 Aalto, Finland\\
\it \small timo.i.takala@aalto.fi, ORCID\/\textup{:} 0009-0006-3376-9635
}

\date{Preliminary version, \today}
\date{}
\date{{To appear in} \emph{J. Reine Angew. Math.}}

\maketitle

\noindent{\small
{\bf Abstract}. 
We introduce a new sphericalization mapping for metric spaces that is applicable 
in very general situations, including  totally disconnected fractal type sets.
For an unbounded  complete metric space
which is uniformly perfect at a base point 
for large radii and    
equipped
with a doubling measure, we make a more specific construction based on the measure and
equip it with a weighted measure.
This mapping is then shown to preserve the doubling property of the measure and 
the Besov (fractional Sobolev) energy.
The corresponding results for flattening of bounded complete metric spaces are also
obtained. Finally, it is shown that for the composition of 
a sphericalization with a flattening, or
vice versa, the obtained space is biLipschitz equivalent with the original space
and the resulting measure is comparable to the original measure.
}

\medskip

\noindent {\small \emph{Key words and phrases}:
Besov energy,
doubling measure,
flattening,
fractional Sobolev energy,
metric space,
sphericalization.
}

\medskip

\noindent {\small \emph{Mathematics Subject Classification} (2020):
Primary:
30L10 
Secondary:  
31E05 
46E36. 
}



\section{Introduction}

The stereographic projection and its inverse between  the Riemann sphere
and the complex plane are very useful mappings in many situations. 
Similar mappings on metric spaces have been studied and used during the last 25 years.
Since the inverse of the stereographic projection ``sphericalizes'' the complex plane, such mappings
from an unbounded to a bounded metric space are called \emph{sphericalizations}.
In the other direction, from a bounded to an unbounded space, they are
called \emph{flattenings}.

These mappings (with suitably chosen parameters) can preserve the Dirichlet energy and the notion
of \p-harmonic functions, and thus make it possible to transform boundary value problems 
in  unbounded  domains to problems in bounded domains, see e.g.\ 
Bj\"orn--Bj\"orn--Li~\cite{BBLi} and 
Gibara--Korte--Shan\-mu\-ga\-lin\-gam~\cite{GibaraKorteShan}.
Natural spaces for \p-harmonic functions are the Sobolev spaces, 
whose traces on the boundary of sufficiently nice domains $\Omega$ 
are the Besov (fractional Sobolev) spaces $B^\theta_p(\bdy\Om)$,
see e.g.\ \cite{BBShyptrace}, \cite{JW84} and~\cite{MalyBesov}.
Besov spaces are also important for  
nonlocal (fractional) partial differential equations and in other situations.

In this paper we therefore study preservation of the Besov
(fractional Sobolev)
 energy
\begin{equation} \label{eq-Besov-seminorm-intro} 
[u]_{\theta,p}^{p}=
[u]_{\theta,p,Z,d,\nu}^{p}:=
\int_Z\int_Z \frac{|u(x)-u(y)|^p}{d(x,y)^{\theta p}} 
      \frac{d\nu(y)\, d\nu(x)}{\nu(B(x,d(x,y)))} 
\end{equation}
under suitably defined sphericalizations and flattenings in rather general
complete metric measure spaces $(Z,d,\nu)$.
In particular, our general results apply to 
closed uniformly perfect subsets of $\R^n$, including many fractals.

Roughly speaking, by a \emph{sphericalization} we mean a
topology-preserving change of metric which makes a complete
unbounded metric space into a bounded space whose completion adds
exactly one point (denoted by $\binfty$) 
which is the accumulation point of every unbounded set in the original space.
Similarly, a \emph{flattening} is a topology-preserving change
of metric 
which makes a punctured complete bounded metric space 
into an unbounded complete metric space such that the puncture
at the base point is sent to infinity.
See Section~\ref{sect-spher-fl-inv} for precise definitions.

The following are some of our main results.
In order to treat sphericalization and flattening simultaneously
we introduce the parameter 
$m_0=0$ for flattening and $m_0=1$ for sphericalizing.
The resulting space is denoted $(\Zhat,\dhat)$.
To be precise, a point $\binfty$ is added to the sphericalized space 
to make it complete,
while a base point $b$ is removed from the original space during flattening.

\begin{thm} \label{thm-intro-doubling+Besov}
Assume that $(Z,d)$
is a complete metric space that is
uniformly perfect at  a base point $b\in Z$
for radii $r\ge m_0$ and that the measure $\nu$ is doubling.
Let the deformed space $(\Zhat,\dhat,\nuhat)$ be defined 
using \eqref{eq-def-dhat-intro}--\eqref{eq-nuhat-intro} below.
\begin{enumerate}
\item \label{it-doubling}
In both cases, sphericalization and flattening, the transformed
measure $\nuhat$ is doubling.
\item \label{it-Besov}
Let $1\le p<\infty$ and $\theta>0$.
If the parameter $\s$ 
in \eqref{eq-rho-spher-flat-intro}--\eqref{eq-nuhat-intro}
satisfies $\s=p\theta$,
then for every measurable function $u$ on $Z$, 
the  Besov\/ \textup(fractional Sobolev\/\textup) energies~\eqref{eq-Besov-seminorm-intro}
with respect to $(Z,d,\nu)$ and $(\Zhat,\dhat,\nuhat)$
 are comparable,
i.e.\ 
\[
     C_1 [u]_{\theta,p,Z,d,\nu}^{p} \le  [u]_{\theta,p,\Zhat,\dhat,\nuhat}^{p} \le C_2[u]_{\theta,p,Z,d,\nu}^{p} 
\]
where the comparison constants $0 < C_1 \le C_2$
depend only on $\s=p\theta$, the doubling constant $C_\nu$ and
the uniform perfectness constant  $\ka$.
\end{enumerate}
\end{thm}

Also inversions are covered by our results, see Remark~\ref{rmk-inversions}.
Roughly speaking, an inversion is  a  topology-preserving change of metric 
between two punctured complete unbounded metric spaces
such  that the puncture at the base point
is sent to infinity while infinity is sent to a new puncture (denoted by $\binfty$).
See Section~\ref{sect-spher-fl-inv} for the precise definition.

In  Section~\ref{sect-duality} we end the paper by showing that
sphericalization and flattening are each other's inverses, up to 
a biLipschitz equivalence of the metric and the measure.

As far as we know, Bonk--Kleiner~\cite{BonkKleiner02} were the first 
to study sphericalization of general metric spaces. 
They used the sphericalized distance
\begin{equation}  \label{eq-def-dhat-BK}
  \dhat(x,y) := \inf \biggl\{ \sum_{j=1}^n \rho(x_j)\rho(x_{j-1})d(x_j,x_{j-1})\biggr\},
\end{equation}
where the infimum is taken over all chains $x=x_0,x_1,\dots,x_n=y$, and with the specific choice 
\[
    \rho(x)=\frac{1}{1+d(x,b)},
\]
where $b$ is a fixed base point.

Some years later, Buckley--Herron--Xie~\cite{BuckleyHerronXie} 
used~\eqref{eq-def-dhat-BK} with $\rho(x)=1/d(x,b)$, 
which led them to 
inversions and flattenings (both of which were called inversions 
in \cite{BuckleyHerronXie}).
Balogh--Buckley~\cite{BaloghBuckley}
used general sphericalizing/flattening functions 
and defined
the new metric by integrating and taking the infimum  
over rectifiable paths between two points.
This way they showed that the quasihyperbolic metrics corresponding 
to the 
spherical and the flat
metrics are biLipschitz equivalent.
Such an approach requires pathconnected spaces.

We have taken a different route, using the sum $\rho(x_j)+\rho(x_{j-1})$ 
in \eqref{eq-def-dhat-intro} below
rather than the product $\rho(x_j)\rho(x_{j-1})$ in~\eqref{eq-def-dhat-BK},
since the product definition does not work with all the parameters we need,
see Remark~\ref{rmk-other-approaches}.
Starting with a general nonincreasing \emph{metric density} function 
$\rho:[0,\infty) \to (0,\infty]$
and a complete metric space $Z=(Z,d)$ 
we define a deformed metric $\dhat$, by letting
\begin{equation}  \label{eq-def-dhat-intro}
  \dhat(x,y) := \inf \biggl\{ \sum_{j=1}^n (\rho(x_j)+\rho(x_{j-1}))d(x_j,x_{j-1}) \biggr\},
\end{equation}
where the infimum is taken over all chains $x=x_0,x_1,\dots,x_n=y$.
Unlike~\cite{BaloghBuckley}, which also considers general sphericalizing functions, 
we do not use any curves and therefore allow disconnected spaces.
Even in connected spaces, our transformations differ from~\cite{BaloghBuckley}.
For example, the completion of the sphericalized $\R$ in~\cite{BaloghBuckley} 
adds two points (corresponding to $\pm \infty$)
and is therefore not a sphericalization in our sense,
 while in our construction it only adds one point $\binfty$,
by the following result for general metric density functions $\rho$.

\begin{prop}       \label{prop-bounded-onepoint-intro}
\quad \begin{enumerate}
\item \label{in-a}
\textup{(Sphericalization)}
Let $(Z,d)$ be complete, unbounded and  uniformly  perfect at large scales at $b$
and $\rho(0)<\infty$. 
Then $(Z,\dhat)$ is bounded and its 
completion adds exactly one point, called $\binfty$, if and only if 
\[
\int_{1}^\infty \rho \,dt < \infty.
\]
\item \label{in-b}
\textup{(Flattening)}
Let $(Z,d)$ be complete, bounded and uniformly perfect at $b$.
Then
$(Z \setminus \{b\},\dhat)$ is unbounded and complete
if and only if 
\[
\int_{0}^{1} \rho \,dt = \infty.
\]
\end{enumerate}
\end{prop}

For most of our results, we assume that $Z$
 is uniformly perfect at the  base point {$b \in Z$}. 
This assumption is needed to guarantee that the distances in the deformed space 
are not larger than expected.
Propositions~\ref{prop-unif-at-infty} and~\ref{prop-unif-at-b'} show that uniform 
perfectness is also essentially preserved by our
general transformations.

In order to show the preservation of measure-theoretic properties,
as in Theorem~\ref{thm-intro-doubling+Besov},
we consider particular choices of the metric density function~$\rho$ 
in~\eqref{eq-def-dhat-intro}, namely
\begin{equation} \label{eq-rho-spher-flat-intro} 
    \rho(t) = \frac{1}{(t+m_0)\nu(B(b,t+m_0))^{1/\s}},
\end{equation}
with $m_0=0$ for flattening and $m_0=1$ for sphericalizing.
In both cases, $\s>0$ is a fixed but arbitrary
parameter and the deformed space $(\Zhat,\dhat)$
is equipped  with the measure $\nuhat$ defined by
\begin{equation} \label{eq-nuhat-intro}
d\nuhat=\rho^\s\, d\nu.
\end{equation}

The papers~\cite{BaloghBuckley}, \cite{BonkKleiner02} and~\cite{BuckleyHerronXie} 
studied purely metric notions, and did not consider any measure.
For the metric transformation~\eqref{eq-def-dhat-BK} considered in \cite{BonkKleiner02},
Wildrick~\cite[Proposition~6.13]{Wildrick} showed that if $(Z,d)$ is connected and
Ahlfors $Q$-regular (with respect to the $Q$-dimensional Hausdorff measure), 
then its sphericalization~$(\Zhat,\dhat)$ 
is also Ahlfors $Q$-regular.
Li--Shan\-mu\-ga\-lin\-gam~\cite{LiShan} considered more general measures 
in connection with sphericalization and flattening of metric spaces and showed that (under suitable conditions) these transformations preserve the doubling property, Ahlfors regularity and Poincar\'e inequalities. 
Further results in this direction were obtained by Durand-Cartagena--Li~\cite{DL1} and \cite{DL2}.

Using a weighted version of the measure 
from~\cite{LiShan}, Bj\"orn--Bj\"orn--Li~\cite{BBLi} showed that (again under suitable assumptions) sphericalization preserves \p-harmonic functions. 
This was then applied to  study the Dirichlet problem and boundary regularity for such functions on unbounded open sets.
All the above papers 
(\cite{BBLi}, \cite{DL1}, \cite{DL2}, \cite{LiShan}  and~\cite{Wildrick})
used the sphericalization and flattening functions from~\cite{BonkKleiner02} and~\cite{BuckleyHerronXie}. 

Gibara--Shan\-mu\-ga\-lin\-gam~\cite{GibaraShan} considered a sphericalization similar 
to the one in \cite{BaloghBuckley}, but with
 the sphericalizing function depending on the distance to the boundary 
of a uniform domain, instead of 
a single base point. They showed that uniformity is preserved in this type of sphericalization.
This same idea was further developed by 
Gibara--Korte--Shan\-mu\-ga\-lin\-gam~\cite{GibaraKorteShan}, where they also defined a new measure on the deformed space. 
They showed that the transformed measure is doubling and the transformed domain supports a Poincar\'e inequality, and used this to solve the \p-Dirichlet problem for 
Besov boundary data.

Most of these papers rely on pathconnectedness of the original space 
(or even stronger assumptions
such as Poincar\'e inequalities).
Due to the nonlocal nature of the Besov 
space this is not required here.
We instead assume the much weaker assumption of uniform perfectness at $b$
(for radii $r \ge m_0$).
In fact, many
fractal type sets are natural objects for our study. 
The smoothness exponent $\theta$ in the Besov
energy \eqref{eq-Besov-seminorm-intro} on 
such sets can even be greater than 1.
In doubling metric measure spaces supporting a certain Sobolev--Poincar\'e inequality,
such as $\R^n$, this would force functions in $B^\theta_p$ 
(defined by \eqref{eq-def-Besov-spc})
to be constant (see Theorem~1.6 and Corollary~1.8 in Kumagai--Shan\-mu\-ga\-lin\-gam--Shimizu~\cite{KSS25}).

At the same time, there are plenty of spaces (even rectifiably connected ones)
for which the Besov spaces $B^\theta_p$ are nontrivial for some $\theta>1$.
More precisely, letting 
\begin{alignat*}{2}
   \theta_p & = \sup\{\theta: 
   B^\theta_{p} \text{ contains nonconstant functions} \} &\quad \text{and}   \\
  s_p & = \sup\{\theta: 
   B^\theta_{p,\infty} \text{ contains nonconstant functions} \},
\end{alignat*}
where $B^\theta_{p,\infty}$ is the  Besov--Lipschitz (Korevaar--Schoen) space,
it follows from \eqref{eq-comp-B-KS} below that 
$\theta_p = s_p$.
For the Sierpi\'nski carpet,
Theorem~1.4 in Murugan--Shimizu~\cite{MurShi} shows that 
\[
s_p= \frac{d_{w,p}}{p}, \quad \text{where $d_{w,p}$ is the \p-walk dimension.}
\]
At the same time, $d_{w,p}>p$ by 
Theorem~2.27 in Shimizu~\cite{Shimizu} 
(or the more general
Theorem~9.8 in Kajino--Shimizu~\cite{KajShi}), and hence $\theta_p>1$
for the Sierpi\'nski carpet.
Similar observations for the Sierpi\'nski gasket follow
from Kajino--Shimizu~\cite[Theorem~5.26]{KajShiToh}
and~\cite[Theorem~9.13]{KajShi}.

Other examples of spaces carrying nontrivial Besov spaces with $\theta>1$ can be
obtained by a general snowflaking process, such as the von Koch snowflake.
(Note that on e.g.\ $\Rn$,
Besov spaces with $\theta >1$ are usually defined in a different way.
Here we ignore such Besov spaces, and only consider
Besov spaces defined as in~\eqref{eq-def-Besov-spc}.)
See Kumagai--Shan\-mu\-ga\-lin\-gam--Shimizu~\cite{KSS25} for recent 
results about potential theoretic implications 
of finite dimensionality of Besov spaces when $\theta>1$.

The walk dimension mentioned above
is an important concept in the differential calculus on fractal sets.
For $p=2$, it describes the scaling $\text{time}\simeq \text{space}^{\beta}$ for random walks or  diffusion processes,
based on Dirichlet forms.
In $\Rn$ the walk dimension $\beta=2$, but on fractal sets typically $\beta>2$.

The outline of the paper is as follows.
In Section~\ref{sect-prelim} we introduce some basic notions essential in this paper.
In particular uniform perfectness, doubling measures and the Besov (semi)norm
are defined here.
In Section~\ref{sect-general} we prove results about the metric $\dhat$ defined 
with a general metric density function $\rho$ in~\eqref{eq-def-dhat-intro}.
In particular we prove Proposition~\ref{prop-bounded-onepoint-intro}.
Here the measure does not play any role.

In Section~\ref{sect-both} we turn to the more specific choices of the metric density
function $\rho$ given by \eqref{eq-rho-spher-flat-intro} 
for the sphericalization and flattening. 
Using some additional notation, we are able to treat both cases simultaneously in
most of our estimates. 
These are mainly carried out in Sections~\ref{sect-both} (where uniform perfectness is not assumed) and~\ref{sect-with-uniform} (where uniform perfectness is made a standing assumption)
leading to the proof of Theorem~\ref{thm-intro-doubling+Besov}\ref{it-doubling} 
(= Theorem~\ref{thm-doubl-m(x)}).
In Section~\ref{sect-Besov-energy-both} we deduce 
Theorem~\ref{thm-intro-doubling+Besov}\ref{it-Besov} (= Theorem~\ref{besov-preserved-both}).
We   
end the paper, in  Section~\ref{sect-duality}, by showing that
if we flatten the sphericalized space, or sphericalize the flattened space,
then the resulting space is biLipschitz equivalent to the original space and the resulting
measure is comparable to the original measure.

\begin{ack}
AB and JB were supported by the Swedish Research Council, grants 2018-04106, 2020-04011 and 2022-04048,
and also by SVeFUM.
RK was supported by the Research Council of Finland grant 360184.
TT was supported by The Magnus Ehrnrooth foundation.
Part of the research was done while RK, SR and TT visited the Department of Mathematics at Link\"oping University. They wish to thank the university for kind hospitality.
We also thank the anonymous referee for a careful reading and useful suggestions.
\end{ack}

\section{Preliminaries}
\label{sect-prelim}

\emph{Throughout the paper we let $Z=(Z,d)$ be a complete metric space 
with $\diam Z > 0$.
We also fix a base point $b \in Z$.
Further standing assumptions are added at the beginning of 
various sections.
}

\medskip

We will use the following notation: 
\begin{equation} \label{eq-Rinfty}
    |x| = d(x,b)
\quad \text{and} \quad
R_\infty  = \sup_{x \in Z} |x|,
\end{equation}
and for balls, 
\[
   B(x,r)  = \{ y \in Z : d(x,y) < r \}
\quad \text{and} \quad
B_r = B(b,r).
\]

All balls considered in this paper are open.
It will be convenient to use this notation also with $r=0$, i.e.\
$B_0 = B(b,0) = \emptyset$.

Throughout the paper, we write $a \simle b$ and $b \simge a$ 
if there is an implicit comparison constant $C>0$ such that $a \le Cb$, 
and $a \simeq b$ if $a \simle b \simle a$.
The implicit comparison constants are allowed to depend on the 
fixed data.
We will carefully explain the dependence in each case.

\subsection{Sphericalizations, flattenings and inversions}
\label{sect-spher-fl-inv}

Our main aim in this paper is to study certain
sphericalizations and flattenings given by a gauge function
as in~\eqref{eq-def-dhat}, 
but we start  by defining what we mean by sphericalizations and flattenings in general.

By a \emph{sphericalization} we mean a
topology-preserving change of metric from an unbounded  complete space $(Z,d)$ to 
a bounded space $(Z,d')$ whose  completion adds
exactly one point $\binfty$,  
so that for any sequence $\{x_j\}_{j=1}^\infty$ in $Z$,

\begin{equation} \label{eq-binfty}
\lim_{j \to \infty}d'(x_j,\binfty) = 0 
\quad \textrm{if and only if} \quad
\lim_{j \to \infty} d(x_j,b) = \infty.
\end{equation}
The metric $d'$ extends directly to $Z \cup \{\binfty\}$.
With a slight abuse of terminology, 
we call also the completion $(Z \cup \{\binfty\},d')$ 
a \emph{sphericalization} of $(Z,d)$.

Next, let us continue to assume that $(Z,d)$ is a complete metric space and let
 $Z_0=Z \setm \{b\}$.
A  \emph{flattening} is a topology-preserving change
of metric from a bounded punctured  space $(Z_0,d)$ 
to  an unbounded complete space $(Z_0,d')$ 
(called a \emph{flattening} of $(Z,d)$)
such that for  any $b_0 \in Z_0$ and any sequence $\{x_j\}_{j=1}^\infty$ in $Z_0$,
\begin{equation} \label{eq-b}
\lim_{j \to \infty}d'(x_j,b_0)  = \infty
\quad \textrm{if and only if}\quad
\lim_{j \to \infty}d(x_j,b)  = 0. 
\end{equation}

Similarly, an \emph{inversion} is  a topology-preserving change
of metric from an unbounded punctured  
space $(Z_0,d)$ to an unbounded  space $(Z_0,d')$ 
whose  completion adds exactly one point $\binfty$
such that
\eqref{eq-binfty} and \eqref{eq-b} hold
for all
sequences in $Z_0$.

Note that $(Z,d)$ is assumed to be complete but not necessarily
proper (i.e.\ its closed bounded subsets need not be compact), 
in which case $(Z \cup \{\binfty\},d')$ will not be compact.
However, from Section~\ref{sect-both} onwards we assume that $(Z,d)$ is
equipped with a doubling measure $\nu$, which together with the completeness
implies that
$(Z,d)$ is proper, by e.g.\ Proposition~3.1 in Bj\"orn--Bj\"orn~\cite{BBbook}.
For the results in Section~\ref{sect-general} it is not even required
that $(Z,d)$ is separable.

\subsection{Uniform perfectness at a point}

The space $Z$ is \emph{uniformly perfect at $x$} 
if there is a constant $\ka>1$ such that
\begin{equation} \label{eq-unif-perfect}
  B(x,\ka r) \setm B(x,r) \ne \emptyset
\quad \text{whenever } B(x,r) \ne Z \text{ and } r>0.
\end{equation}  
Note that the condition $B(x,r) \ne Z$ in 
\eqref{eq-unif-perfect}
can equivalently be replaced by $B(x,\ka r) \ne Z$.
The space $Z$ is \emph{uniformly perfect} if it is uniformly
perfect at every $x$ with the same constant $\ka$.
See Heinonen~\cite[Chapter~11]{Heinonen} 
and Bj\"orn--Bj\"orn~\cite[Section~2]{BBbring}
for further discussion and history.
Note that $Z$ is uniformly perfect with any $\ka>1$ if it is connected.
We will also need the following notions, 
which may not have been considered before.

Let 
$m_0 \ge 0$.
Then $Z$ is \emph{uniformly perfect at $x$ 
for radii $r\ge m_0$} if there is a constant $\ka>1$ such that
\begin{equation*} 
  B(x,\ka r) \setm B(x,r) \ne \emptyset
\quad \text{whenever } B(x,r) \ne Z \text{ and } 0 \ne r \ge m_0.
\end{equation*}
Note that $Z$ is uniformly perfect at $x$ if and only if it is 
uniformly perfect at $x$  for radii $r\ge 0$.

We also say that $Z$ is \emph{uniformly perfect at large scales at $x$}, if it is unbounded and uniformly perfect at $x$ for radii $r\ge 1$.
It is easy to see that the constant $1$ in $r\ge1$ can equivalently be replaced by any other positive number, provided that $\ka$ is allowed to change.

\subsection{Doubling measures} 
\label{sect-doubl}

In this section, we will 
provide some preliminary estimates for a doubling
 measure under the assumption that the space is uniformly perfect 
at a fixed base point $b\in Z$.

A complete Borel regular measure $\nu$  is \emph{doubling}
if  there is a \emph{doubling constant} $C_\nu>1$ such that
\begin{equation*} 
0 < \nu(B(x,2r)) \le C_\nu \nu(B(x,r)) < \infty
\quad \text{for all $x \in Z$ and $r>0$}. 
\end{equation*}

The following simple lemma from Bj\"orn--Bj\"orn~\cite[Lemma~3.6]{BBbook} 
will be used several times.

\begin{lem} \label{lem-comp-close-balls}
Assume that $\nu$ is a doubling measure on $Z$.
Let $B=B(x,r)$ and $B'=B(x',r')$ be two balls such that
\[
d(x,x')\le ar
\quad \text{and} \quad  
\frac{r}{a}\le r'\le ar.
\]
Then $\nu(B) \simeq \nu(B')$ with comparison constants
depending only on $a \geq 1$ and $C_\nu$.
\end{lem}

The following consequence of the doubling property and uniform perfectness is a (strong) type of reverse-doubling.
Recall from \eqref{eq-Rinfty}
that $R_\infty  = \sup_{x \in Z} |x|$.

\begin{lem} \label{lem-measure-growth}
Assume that $\nu$ is a doubling measure on $Z$
and that $Z$ is 
uniformly perfect at $b$ for radii $r\ge m_0$
with constant~$\kappa$.
Then there are constants $\alp>0$ and $\Lambda \ge 1$,
only depending on $C_\nu$ and $\ka$, such that
\begin{equation*} 
  \frac{\nu(B_r)}{\nu(B_R)} \le \Lambda \Bigl(\frac rR\Bigr)^\alp
  \quad \text{when } m_0 \le  r < R <  2R_\infty.
\end{equation*}
In particular $\nu(\{ b \}) = 0$ if $m_0=0$.
\end{lem}

\begin{proof}
If $R_\infty \leq 2 m_0$, then
\begin{equation*}
\frac{\nu(B_r)}{\nu(B_R)}
\leq 1
\leq \frac{4r}{R},
\end{equation*}
so the claim holds with $\alpha = 1$ and $\Lambda = 4$.

Now assume that $R_\infty > 2 m_0$. 
The claim is trivial if $r=0$ since $B_0=\emptyset$, so we may assume that
$r>0$.
Let $t>0$ be such that
$m_0 < t < \tfrac12 R_\infty$.
By the uniform perfectness, there is $x \in B_{2\ka t} \setm B_{2t}$.
Then 
\[
  \nu(B_t) \le \nu(B(x,3 \ka t)) \simle \nu(B(x,t)).
\]
Thus
\begin{equation} \label{eq-M}
  \nu(B_{4\ka t}) \ge \nu(B_t) + \nu(B(x,t))
  \ge \Lambda \nu(B_t),
\end{equation}
where $\Lambda >1$.

For $j=0,1,\dots$\,, let $r_j=(4\ka)^{-j} R$ and find $n$ such that 
$r_{n+1} \le  r < r_{n}$.
Note that $r_n>m_0$ and $r_1 < \tfrac12 R_\infty$.
Then, by \eqref{eq-M},
\[
  \frac{\nu(B_r)}{\nu(B_R)}
  \le     \frac{\nu(B_{r_n})}{\nu(B_{R})}
  = \prod_{j=1}^n    \frac{\nu(B_{r_{j}})}{\nu(B_{r_{j-1}})}
  \le \Lambda^{-n}
  = \Lambda \Bigl(\frac{r_{n+1}}{R}\Bigr)^\alp
  \le \Lambda \Bigl(\frac rR\Bigr)^\alp,
\]
where 
\[
  \alp = \frac{\log \Lambda}{\log 4\ka}.
  \qedhere
\]
\end{proof}

\subsection{The Besov (fractional Sobolev) energy}
\label{sect-Besov} 

Let $1 \le p< \infty$ and $\theta >0$.
For a measurable function $u: Z \to [-\infty,\infty]$ (which is
finite $\nu$-a.e.) we define the
\emph{Besov\/ {\rm(}fractional Sobolev\/{\rm)} 
energy {\rm(}\/seminorm{\rm)}} 
by
\begin{equation} \label{eq-Besov-seminorm} 
  [u]_{\theta,p}^{p}=
[u]_{\theta,p,Z}^{p}=
 \int_Z\int_Z \frac{|u(x)-u(y)|^p}{d(x,y)^{\theta p}} 
      \frac{d\nu(y)\, d\nu(x)}{\nu(B(x,d(x,y)))}. 
\end{equation}
Here and elsewhere, the  integrand 
should be interpreted as zero  when $y=x$.

The \emph{Besov\/ \textup(fractional Sobolev\/\textup) space} 
$B^\theta_p(Z)$ 
consists of functions $u$ for which the norm
\begin{equation}   \label{eq-def-Besov-spc}
\|u\|_{B^{\theta}_p (Z)}
:=[u]_{\theta,p}+\|u\|_{L^p(Z)} 
\end{equation}
is finite.
This is a Banach space (after taking $\nu$-a.e.-equivalence classes),
see Remark~9.8 in Bj\"orn--Bj\"orn--Shan\-mu\-ga\-lin\-gam~\cite{BBShyptrace}.
We restrict our attention to Besov spaces with two indices (i.e.\ ``$q=p$'').

The above spaces and (semi)norms go under various names, 
such as fractional Sobolev, 
Gagliardo--Nirenberg, Sobolev--Slobodetski\u{\i} and Besov.
Besov spaces is the usual name in the metric space literature,
while fractional Sobolev spaces is more common in e.g.\  the community
working with nonlinear nonlocal operators on $\R^n$ 
such as the fractional \p-Laplacian $(-\Delta_p)^s$.

Equivalent definitions, using equivalent seminorms, can be found in
Gogatishvili--Koskela--Shan\-mu\-ga\-lin\-gam~\cite[Theorem~5.2 and (5.1)]{GKS} 
for the case when $\nu$ is doubling.
When $\nu$ is also reverse-doubling (or equivalently $Z$ is uniformly perfect),
further equivalent definitions can be found in
Gogatishvili--Koskela--Zhou~\cite[Theorem~4.1 and Proposition~4.1]{GKZ}, 
for example that
the Besov space $B^\theta_p(Z)$ considered here
coincides with the corresponding Haj\l asz--Besov space.
By \cite[Lemmas~6.1 and~6.2]{GKS}, it is 
related to fractional
Haj\l asz spaces, considered already in Yang~\cite{Yang2003}.

The spaces $B^\theta_p(Z)$ are  closely related to the 
so-called Korevaar-Schoen (or Besov--Lipschitz) spaces
$B^\theta_{p,\infty}(Z)$ considered in the literature on fractals,
see e.g.\ 
\cite{KajShiToh}, \cite{KSS25} and~\cite{MurShi}.
More precisely, if $\nu$ is doubling,
then for every $0<\de<\theta$,
\begin{equation}    \label{eq-comp-B-KS}
B^\theta_p(Z) \subset B^\theta_{p,\infty}(Z)  \subset B^{\theta - \de}_p(Z),
\end{equation}
see Kumagai--Shan\-mu\-ga\-lin\-gam--Shimizu~\cite[Lemma~2.6]{KSS25}
(where it also shown that $B^\theta_{p,\infty}(Z)=KS_p^\theta(Z)$ when
$\nu$ is doubling).

Our definition~\eqref{eq-def-Besov-spc} is also equivalent to certain 
norms based on heat kernels (with ``$q=p$''), under 
suitable a priori estimates for the kernel,  see 
Saloff-Coste~\cite[Th\'eor\`eme~2]{SF1990} (on Lie groups)
and Pietruska-Pa\l uba~\cite[Theorem~3.1]{KasiaP2010} (on metric spaces).
For $p=2$, our definition of $[u]_{\theta,2}^2$ coincides with
the energy used in connection with heat kernel estimates
in Chen--Kumagai~\cite{ChenKumagai08} 
when
\[
    J(x,y)=\frac{1}{\mu(B(x,d(x,y))\rho(x,y)^{2\theta}}
\]
therein.
See the above papers for the precise definitions and
earlier references to the theory on $\R^n$, 
 fractals and  Ahlfors regular metric spaces.

\section{Deformation of the metric}
\label{sect-general}

\emph{Recall the standing assumptions from the beginning of Section~\ref{sect-prelim}.
In particular, $Z=(Z,d)$ is a complete metric space  throughout the paper.}

\medskip

In this section we deal with purely metric notions, and do not need any
measure.

Let $\rho:(0,\infty)\to (0,\infty)$ be a nonincreasing function
and let $\rho(0)=\lim_{t\to0\limplus}\rho(t)$.
Also let $\rho(x)=\rho(|x|)$ for $x \in Z$.
We define
\begin{equation}  \label{eq-def-dhat}
  \dhat(x,y) := \inf \biggl\{ \sum_{j=1}^n (\rho(x_j)+\rho(x_{j-1}))d(x_j,x_{j-1}) \biggr\},
\end{equation}
where the infimum is taken over all chains $x=x_0,x_1,\dots,x_n=y$.
We call $\rho$ a \emph{metric density} function.

If $\rho(0)=\infty$, then $\dhat(x,b)=\infty$ for every $x \ne b$, 
and therefore $\dhat$ is not a metric on $Z$.
We therefore introduce 
\[
    Z' =  \{x\in Z: \rho(x)<\infty\} =
    \begin{cases}
      Z, & \text{if } \rho(0)<\infty, \\
      Z \setm \{b\}, & \text{if } \rho(0)=\infty.
    \end{cases}
\]

We will show in Proposition~\ref{prop-dhat-metric} that $\dhat$ is a metric on $Z'$.
The completion of $(Z',\dhat)$ will be denoted by $\Zhat$. 
The metric $\dhat$ extends directly to $\Zhat$.
Lemma~\ref{lem-sari} below implies that if 
$Z$ is uniformly perfect at $b$,
$\rho(0)=\infty$ and 
$\rho\in L^1(0,1)$, 
then $\Zhat$ adds back $b$ to $Z'$
with
\begin{equation*} 
\dhat(b,x):=\dhat(x,b):=  \lim_{y\to b} \dhat(x,y).
\end{equation*}
Balls with respect to the metric $\dhat$ are denoted by 
\[
\Bhat(x,r) := \{ y \in \Zhat: \dhat(x,y) < r\}.
\]

\begin{prop} \label{prop-dhat-metric}
  $\dhat$ is  a metric on $Z'$.
\end{prop}

For proving this proposition,
 we will need the following lemma, which will be useful also later.

\begin{lem} \label{lem-m-M}
Let $x,y \in Z'$ be such that $x \neq y$.
Then 
\begin{equation} \label{eq-m-M}
0 < d(x,y) \inf_{B(x,d(x,y))} \rho \le \dhat(x,y) \le (\rho(x)+\rho(y)) d(x,y).
\end{equation}
\end{lem}

\begin{proof}
Let
$x=x_0,x_1,\dots,x_n=y$ be any chain from $x$ to $y$ and
let $j_0$ be the smallest index such that $x_{j_0}\notin B(x,d(x,y))$.
(Such a $j_0$ always exists as $x_n=y \notin B(x,d(x,y))$.)
Let $m= \inf_{B(x,d(x,y))}\rho>0$.
Then 
\begin{align*}
\sum_{j=1}^n (\rho(x_j)+\rho(x_{j-1}))d(x_j,x_{j-1})
& \ge \sum_{j=1}^{j_0 } \rho(x_{j-1})d(x_j,x_{j-1}) \\
& \ge m\sum_{j=1}^{j_0 }d(x_j,x_{j-1})
\ge m d(x, x_{j_0})\ge m d(x,y).
\end{align*}
Taking the infimum over all chains from $x$ to $y$ shows 
the first inequality in~\eqref{eq-m-M}.

For the second inequality, we instead use the trivial chain
$x=x_0,x_1=y$ to obtain that 
\[
   \dhat(x,y) \le (\rho(x)+\rho(y)) d(x,y). 
\qedhere
\]
\end{proof}

\begin{proof}[Proof of Proposition~\ref{prop-dhat-metric}]
Clearly $\dhat(x,x)=0$ for each $x \in Z'$, and $\dhat$ is symmetric and satisfies
the triangle inequality.
Finally, $0 <  \dhat(x,y)< \infty $
for all $x,y \in Z'$ with $x \ne y$, 
by Lemma~\ref{lem-m-M}.
\end{proof}

\begin{lem} \label{lem-d(x,y)-with-int}
Assume that  $x,y \in Z$ and that $|x|\le|y|$.
Then
\[
    \dhat(x,y) \ge \int_{|x|}^{|y|} \rho\,dt.
\]  
\end{lem}

\begin{proof}
Let
$x=x_0, x_1, \ldots, x_n=y$ be a chain.
Then
\begin{align*}
\sum_{j=1}^n (\rho(x_j)+\rho(x_{j-1}))d(x_j,x_{j-1})
&\ge \sum_{|x_j|\ge|x_{j-1}|} \rho(x_{j-1})d(x_j,x_{j-1})  \nonumber \\
  &\ge \sum_{|x_j|\ge|x_{j-1}|} \int_{|x_{j-1}|}^{|x_{j}|} \rho\,dt
\ge \int_{|x|}^{|y|} \rho\,dt.
\end{align*}
The claim then follows by taking the infimum over all chains from $x$ to $y$.
\end{proof}

\begin{prop} \label{prop-top-pres}
The deformation $(Z',d) \mapsto (Z',\dhat)$ is topology-preserving, i.e.\ 
for any sequence $\{x_j\}_{j=1}^\infty$ in $Z'$ and $x \in Z'$,
\begin{equation} \label{eq-top-preserving}
\lim_{j \to \infty} d(x_j,x) = 0
\quad \textrm{if and only if} \quad
\lim_{j \to \infty} \dhat(x_j,x) = 0.
\end{equation}
\end{prop}

\begin{proof}
If $d(x_j,x)\to 0$, then clearly $|x_j| \to |x|$
and so for big enough $j$ we have $|x_j|\ge |x|/2$ and $\rho(x_j)\le \rho(|x|/2)$.
It thus follows from Lemma~\ref{lem-m-M} that
$\dhat(x_j,x)\to 0$.

Conversely, if $\dhat(x_j,x)\to 0$, then 
$|x_j| \to |x|$ by Lemma~\ref{lem-d(x,y)-with-int}. 
So for large enough $j$ we have $|x_j|\le |x|+1$
and hence
$B(x,d(x_j,x))\subset B(b,3|x|+1)$. 
Thus, Lemma~\ref{lem-m-M} implies that for these large enough $j$,
\[
d(x_j,x)\le \frac{ \dhat(x_j,x)}{\rho(3|x|+1)},
\]
yielding $d(x_j,x)\to 0$.
\end{proof}

Using uniform perfectness we can obtain
the following converse of Lemma~\ref{lem-d(x,y)-with-int}.

\begin{lem} \label{lem-sari}
Assume that $Z$ is uniformly perfect at $b$ 
for radii $r\ge m_0$ with constant $\ka > 1$.
Let $x,y \in Z'$ with $m_0 \le  |x| \le |y|$.
Then
\begin{equation*} 
  \dhat(x,y) \simle \int_{|x|/\ka}^{|y|}\rho\,dt,
\end{equation*}
where the comparison constant only depends on $\ka$.
\end{lem}

\begin{proof}
Assume first that $|x|>0$.
By the uniform perfectness
we can find a finite sequence of points
$x=x_0,\dots,x_n=y$ such that
\[
    \ka^{j-1} |x|  \le |x_j| \le \ka^j |x|,
  \quad j=1,2,\dots,n.
\]
Thus
\begin{align*}
  \dhat(x,y)
  & \le \sum_{j=1}^n (\rho(x_j)+\rho(x_{j-1}))d(x_j,x_{j-1})
   \le 4 \sum_{j=1}^{n} |x_j|\rho(x_{j-1})\\
   & \le 4 \ka^2 \biggl(|x|\rho(x) + \sum_{j=2}^{n}\rho(\ka^{j-2}|x|)\ka^{j-2}|x| \biggr) 
    \le \frac{8\ka^3}{\ka -1}\int_{|x|/\kappa}^{|y|}\rho\,dt.
\end{align*}
If $|x| = 0$, then $\rho(0) < \infty$ and $m_0=0$. 
If $|y| = 0$, the claim is trivial, so assume that $|y| > 0$. 
By the uniform perfectness there exists a sequence $\{x_i\}_{i=1}^\infty$ such that $0 < |x_i| \leq |y|$ for every $i$ and $\lim_{i \to \infty} |x_i| = 0$.
Therefore, by the above estimate for $\dhat(x_i,y)$,
\begin{equation*}
\dhat(x,y)
\leq \dhat(x,x_i) + \dhat(x_i,y)
\simle \rho(0) d(x,x_i) + \int_{|x_i|/\ka}^{|y|} \rho \,dt
\to \int_{0}^{|y|} \rho \,dt, \quad \text{as } i\to\infty.
\qedhere
\end{equation*}
\end{proof}

As a consequence of Lemmas~\ref{lem-d(x,y)-with-int} and~\ref{lem-sari} we 
obtain the following characterizations.

\begin{prop}\label{prop-bounded-onepoint}
Assume that $\rho(0)<\infty$ and that $(Z,d)$ is unbounded
and 
 uniformly  perfect at large scales at $b$. 
Then the following are equivalent\/\textup{:}
\begin{enumerate}
\item \label{a-bdd}
$(Z,\dhat)$ is bounded, 
\item \label{a-int}
\[
\int_{1}^\infty \rho \,dt < \infty,
\]
\item \label{it-not-compl}
$(Z,\dhat)$ is not complete,
\item \label{a-compl}
the completion $\Zhat$ of $(Z,\dhat)$ adds exactly one point, 
denoted $\binfty$,
\item \label{a-spher}
$(\Zhat,\dhat)$ is a sphericalization of $(Z,d)$,
i.e.\ it 
adds exactly one point $\binfty$ to $Z$ and 
\eqref{eq-binfty} 
{\rm(}with $d'$ replaced by $\dhat${\rm)
and \eqref{eq-top-preserving} hold}.
\end{enumerate}
\end{prop}

Note that $Z' = Z$, because $\rho(0) < \infty$, and thus $(Z,\dhat)$ is a metric space.
Note also that $\int_0^1 \rho \, dt \leq \rho(0) < \infty$.

\begin{proof} 
\ref{a-bdd}\eqv\ref{a-int}
This follows directly from Lemmas~\ref{lem-d(x,y)-with-int} and~\ref{lem-sari},
as $Z$ is unbounded and so $\sup_{y\in Z} |y|=\infty$.

Next, consider an arbitrary $d$-bounded sequence $\{x_j\}_{j=1}^\infty$ in $Z$.
Then $d(x_j,x_k) \le |x_j|+|x_k|$ and thus, 
$B(x_j,d(x_j,x_k))\subset B(b,3\sup_i|x_i|)$.
Hence,
by Lemma~\ref{lem-m-M},
\begin{equation}   \label{eq-comp-d-dhat-bdd}
  \rho\Bigl(3 \sup_i |x_i|\Bigr)
 d(x_j,x_k) \le    \dhat(x_j,x_k) \le 2\rho(0) d(x_j,x_k).
\end{equation}
It follows that $\{x_j\}_{j=1}^\infty$ is a $\dhat$-Cauchy sequence if and
only if it is a $d$-Cauchy sequence, in which case it also
converges with respect to both metrics, since
$(Z,d)$ is complete.

\ref{a-int}\imp\ref{a-spher} 
Let $\{x_j\}_{j=1}^\infty$  be a $\dhat$-Cauchy sequence.
If it is $d$-bounded, then we have already seen 
by \eqref{eq-comp-d-dhat-bdd} that it converges within $Z$.
On the other hand,  if $\{x_j\}_{j=1}^\infty$ is $d$-unbounded,
then it does not $\dhat$-converge to any $z \in Z$.
Indeed, if $\dhat(x_j,z)\to0$ for some $z\in Z$, 
then by Lemma~\ref{lem-d(x,y)-with-int},
\[
0 <\int_{|z|}^\infty \rho\,dt = 
\limsup_{j\to\infty} \int_{|z|}^{|x_j|} \rho\,dt
\le \limsup_{j\to\infty} \dhat(z,x_j) =0,
\]
which is a contradiction.
We therefore denote the $\dhat$-limit of $\{x_j\}_{j=1}^\infty$
by~$\binfty$.
This also shows that   $\lim_{j\to\infty}|x_{j}|=\infty$,
since any $d$-bounded subsequence of $\{x_j\}_{j=1}^\infty$
would $\dhat$-converge within $Z$ because of \eqref{eq-comp-d-dhat-bdd},
which is impossible.
Similarly,
if $\{y_j\}_{j=1}^\infty$ is another $d$-unbounded $\dhat$-Cauchy sequence, then 
also $\lim_{j\to\infty}|y_{j}|=\infty$.
By mixing $x_j$ and $y_j$, we obtain a new unbounded sequence, which
must be a $\dhat$-Cauchy sequence,
because of Lemma~\ref{lem-sari} and  \ref{a-int}.
Hence every $d$-unbounded $\dhat$-Cauchy sequence converges to $\binfty$,
and $\Zhat=Z \cup \{\binfty\}$.
Moreover, if $\{z_j\}_{j=1}^\infty$ is any sequence
with $\lim_{j \to \infty}|z_j| = \infty$ then it is also $\dhat$-Cauchy by Lemma~\ref{lem-sari} and  \ref{a-int}.
In particular, \eqref{eq-binfty} holds 
{\rm(}with  $d'$ replaced by $\dhat${\rm)}.
Finally, the deformation is topology-preserving on $Z'=Z$
by Proposition~\ref{prop-top-pres} i.e.\ \eqref{eq-top-preserving} holds.

\ref{a-spher}\imp\ref{a-compl}\imp\ref{it-not-compl}
These implications are 
trivial.

$\neg$\ref{a-int}\imp$\neg$\ref{it-not-compl}
Let $\{x_j\}_{j=1}^\infty$  be a $\dhat$-Cauchy sequence.
Then it follows from  Lemma~\ref{lem-d(x,y)-with-int} and $\neg$\ref{a-int} 
that $\{x_j\}_{j=1}^\infty$ is $d$-bounded, and hence by the above it is $\dhat$-convergent within $Z$. 
So $(Z,\dhat)$ is complete.
\end{proof}

\begin{prop}\label{prop-unbounded-complete}
Assume that $(Z,d)$ is bounded
and uniformly perfect at $b$. 
Then the following are equivalent\/\textup{:}
\begin{enumerate}
\item \label{b-unbdd}
$(Z \setminus \{b\},\dhat)$ is unbounded, 
\item \label{b-int}
\[
\int_{0}^{1} \rho \,dt = \infty,
\]
\item \label{b-compl}
$(Z \setminus \{b\},\dhat)$ is complete,
\item \label{b-flat}
$(Z \setminus \{b\},\dhat)$ is a flattening of $(Z \setminus \{b\},d)$,
i.e.\ it is complete and
\eqref{eq-b} 
{\rm(}with  $d'$ replaced by $\dhat${\rm)
and \eqref{eq-top-preserving} hold}.
\end{enumerate}
\end{prop}

Note that \ref{b-int} implies that $\rho(0)=\infty$ and hence 
that $Z'=Z\setm\{b\}$.

\begin{proof}
\ref{b-unbdd}\eqv\ref{b-int} 
This follows directly from Lemmas~\ref{lem-d(x,y)-with-int} and~\ref{lem-sari},
since $Z$ is uniformly perfect at $b$ and hence $\inf_{x\in Z \setminus \{b\}} |x|=0$.

\ref{b-int}\imp\ref{b-flat}
Let $\{x_j\}_{j=1}^\infty$  be a $\dhat$-Cauchy sequence in $Z\setm\{b\}$.
Then it follows from \ref{b-int} and Lemma~\ref{lem-d(x,y)-with-int} that 
$\inf_i |x_i|>0$.
Since $(Z,d)$ is bounded, Lemma~\ref{lem-m-M} implies that
\[
  \rho\Bigl(\sup_{x\in Z} |x|\Bigr) 
 d(x_j,x_k) \le    \dhat(x_j,x_k) \le 2\rho\Bigl(\inf_i |x_i|\Bigr) d(x_j,x_k).
\]
Thus $\{x_j\}_{j=1}^\infty$ is a $d$-Cauchy sequence,
and moreover has the same $\dhat$-limit as the $d$-limit provided
by the completeness of $Z$.
That the deformation is topology-preserving on $Z \setm \{b\}$
follows from Proposition~\ref{prop-top-pres}.

It remains to show \eqref{eq-b}.
To this end, let 
$b_0 \in Z \setm \{b\}$
and let
$\{x_j\}_{j=1}^\infty$ be an arbitrary sequence in 
$Z \setm \{b\}$.
If $|x_j| \to 0$, then it follows from \ref{b-int} and 
Lemma~\ref{lem-d(x,y)-with-int} that
$\dhat(x_j,b_0) \to \infty$.
Conversely, if $\dhat(x_j,b_0) \to \infty$
then it follows from \ref{b-int},
Lemma~\ref{lem-sari} and the boundedness of $(Z,d)$ that
$|x_j| \to 0$.

\ref{b-flat}\imp\ref{b-compl}
This is trivial.

$\neg$\ref{b-int}\imp$\neg$\ref{b-compl}
By the uniform perfectness, there is a sequence $\{x_j\}_{j=1}^\infty$ in 
$Z \setminus \{b\}$
such that $|x_j| \to 0$, as $j \to \infty$.
By Lemma~\ref{lem-sari} and $\neg$\ref{b-int}, 
$\{x_j\}_{j=1}^\infty$ is a $\dhat$-Cauchy sequence.
On the other hand, it follows from 
Proposition~\ref{prop-top-pres}
that $\{x_j\}_{j=1}^\infty$  cannot $\dhat$-converge to 
any $z\ne b$.
\end{proof}

\begin{proof}[Proof of Proposition~\ref{prop-bounded-onepoint-intro}]
Part~\ref{in-a} follows directly from 
Proposition~\ref{prop-bounded-onepoint},
while \ref{in-b} follows directly from
Proposition~\ref{prop-unbounded-complete}.
\end{proof}

Next we will show that uniform perfectness is in a certain sense preserved by 
metric transformations.

\begin{prop} \label{prop-unif-at-infty}
Assume that $\rho(0)<\infty$, that $(Z,d)$ is unbounded and 
 uniformly  perfect at large scales at $b$,
and that $(Z,\dhat)$ is bounded.
Also  assume that there is $A \ge 1$ such that
\begin{equation}   \label{eq-doubl-rho}
\rho(t) \le A\rho(2t) \quad \text{for every $t>0$}.
\end{equation}

Then $(\Zhat,\dhat)$ is uniformly perfect at $\binfty$ 
with a constant $\kahat > 1$ that depends only on $A$, $\ka$ and $\rho(0)/\rho(1)$.
\end{prop}

Condition \eqref{eq-doubl-rho} can be seen as a \emph{doubling condition}
for the nondecreasing function $t\mapsto 1/\rho(t)$.

\begin{proof}
By the uniform perfectness at large scales, there are $x_j \in Z$ 
such that
\[
    \ka^{j-1}  \le |x_j| \le \ka^{j},
  \quad j=1,2,\dots.
\]  
In particular, $\dhat(x_j,\binfty)\to0$ by \eqref{eq-binfty} (from
Proposition~\ref{prop-bounded-onepoint}\,\ref{a-spher}).
By Lemmas~\ref{lem-d(x,y)-with-int} and~\ref{lem-sari}, 
and the doubling property~\eqref{eq-doubl-rho} of $\rho$ we then see that
\begin{equation} \label{eq-23}
    \int_{|x|}^\infty \rho \,dt 
\le \dhat(x,\binfty) 
\simle \int_{|x|/\ka}^{\infty}\rho\,dt
\simeq \int_{|x|}^{\infty}\rho\,dt
  \quad \text{if } |x| \ge 1.
\end{equation}
Since $1  \le |x_1| \le \ka$, the monotonicity of $\rho$ and another use 
of \eqref{eq-doubl-rho} give us that
\begin{equation} \label{eq-24}
    \dhat(x_1,\binfty) \simeq \int_{1}^{\infty}\rho\,dt \ge \rho(2) \simeq \rho(1)
\end{equation}
and that there exists a constant $\ka'>1$ (only depending on $A$ and $\ka$) such that 
\begin{equation*}  
\frac{1} {\ka'} \dhat(x_j,\binfty) \le \dhat(x_{j-1},\binfty) \le \ka' \dhat(x_j,\binfty),
  \quad j=2,3,\dots.
\end{equation*}
Let $0 < r \leq \dhat(x_1,\binfty)$. 
Since
 $\dhat(x_j,\binfty)\to0$, as $j \to \infty$, 
there is a smallest
integer $j_r$
such that $\dhat(x_{j_r+1},\binfty) < r$.
Then $j_r \ge 1$ and
\begin{equation*}
r \leq \dhat(x_{j_r},\binfty) \le \ka'\dhat(x_{j_r+1},\binfty)
< \ka' r. 
\end{equation*}

Next,  \eqref{eq-23} and~\eqref{eq-24} show that
\[
   \dhat(x,\binfty) \simle \dhat(x_1,\binfty) 
\quad \text{if } |x| \ge 1.
\]
On the other hand, if $|x| < 1$ then $d(x,x_1) \le 1+\ka$ and estimating
$\dhat(x,x_1)$ with the trivial chain $x=x_0,x_1$, together with 
the monotonicity of $\rho$ and~\eqref{eq-24}, shows that
\[
   \frac{\dhat(x,\binfty)}{\dhat(x_1,\binfty)}
   \le \frac{\dhat(x,x_1)+\dhat(x_1,\binfty)}{\dhat(x_1,\binfty)}
   \simle \frac{2(1+\ka) \rho(0)}{\rho(1)} +1
   \simle \frac{\rho(0)}{\rho(1)}.
\]
We can therefore find $\kahat \ge \ka'$, depending
only on $A$, $\ka$ and $\rho(0)/\rho(1)$, such that
\[
\kahat \geq \frac{ \sup_{x \in Z} \dhat(x,\binfty)}{\dhat(x_1,\binfty)},
\]
and thus $\Bhat(\binfty,\kahat r)= \Zhat$
when $r > \dhat(x_1,\binfty)$.
Hence, $(\Zhat,\dhat)$ is uniformly perfect at $\binfty$ with constant $\kahat$.
\end{proof}

\begin{prop} \label{prop-unif-at-b'}
Assume that $Z$ is uniformly  perfect at $b$ and that
\begin{equation} \label{eq-unif-at-b'}
\int_0^1 \rho \,dt = \infty.
\end{equation}
Also assume that the doubling condition~\eqref{eq-doubl-rho} for $\rho$ holds.
Then for every $b' \in Z'$, the space $(Z', \dhat)$
is uniformly perfect at large scales at $b'$ with a constant $\kahat > 1$ that depends only on $A$, $\ka$
and $|b'|\rho(b')$.
\end{prop}

\begin{proof}
Note that by \eqref{eq-unif-at-b'} and the monotonicity
of $\rho$, we have $\rho(0) = \infty$ and therefore $Z' = Z \setminus \{b\}$.
By the uniform perfectness at $b$, there are $x_j \in Z'$
such that
\[
     \ka^{-j-1}|b'|  \le |x_j| \le \ka^{-j}|b'|,
   \quad j=1,2,\dots.
\]
By Lemmas~\ref{lem-d(x,y)-with-int} and~\ref{lem-sari},
we see that
\[ 
\int_{|x_j|}^{|b'|} \rho \,dt
 \le \dhat(x_j,b')
\simle \int_{|x_{j}|/\ka}^{|b'|}\rho\,dt 
= \int_{|x_{j}|/\ka}^{|x_j|}\rho\,dt + \int_{|x_j|}^{|b'|}\rho\,dt.
\] 
The doubling condition~\eqref{eq-doubl-rho} of $\rho$ gives that
\[
 \int_{|x_{j}|/\ka}^{|x_j|}\rho\,dt 
\simeq \rho(x_j) |x_j| \simle
\int_{|x_j|}^{|b'|}\rho\,dt
\]
and then similarly, 
\[
\dhat(x_j,b') 
\simeq \int_{|x_j|}^{|b'|} \rho \,dt
\simeq \int_{|x_{j+1}|}^{|b'|} \rho \,dt
 \simeq  \dhat(x_{j+1},b').
 \]
In particular, $\dhat(x_1,b') \simeq |b'|\rho(b')$.
Here the comparison constants depend only on $A$ and $\ka$.
From \eqref{eq-unif-at-b'} we get that $\lim_{j \to \infty} \dhat(x_j,b') = \infty$.
Then for some $\kahat>1$ and each $r\ge \dhat(x_1,b')$ there is $j_r $ such that
\[
x_{j_r} \in \Bhat(b',\kahat r) \setm \Bhat(b',r).
\]
After replacing $\kahat$ by $\kahat \max\{1,\dhat(x_1,b')\}$, the same statement holds for all $r\ge1$.
\end{proof}

\begin{remark} \label{rmk-other-approaches}
We will now 
explain why it is important in
this paper to define  
sphericalization and flattening using
a sum as in \eqref{eq-def-dhat}  rather than 
a product.
Suppose  that one would like to sphericalize an unbounded metric space $(Z,d)$ using
the formula 
\begin{equation}  \label{eq-def-dtilde}
  \dtilde(x,y) := \inf \biggl\{ \sum_{j=1}^n \rhotilde(x_j)\rhotilde(x_{j-1})d(x_j,x_{j-1})\biggr\},
\end{equation}
where the infimum is taken over all chains $x=x_0,x_1,\dots,x_n=y$, 
and with a general positive nonincreasing function $\rhotilde(x)=\rhotilde(|x|)$ 
such that $\rhotilde(0)<\infty$.
Note that a transformation 
using products as in \eqref{eq-def-dtilde}
that is equivalent to a transformation 
using sums as in \eqref{eq-def-dhat},
or path integrals,
would typically have $\rhotilde\simeq\sqrt{\rho}$. 

The main problem with the product definition is that 
$\rhotilde(x_j)\rhotilde(x_{j-1})$ can be small even when $\rhotilde(x_j)$
is large, provided that $x_{j-1}$ is very far away and $\rhotilde$ is fast decreasing.
We can see this by considering a transformation with 
$\rhotilde$ such that $\lim_{t \to \infty} t\rhotilde(t)=0$.
Let $z_1,z_2 \in Z$.
Since $(Z,d)$ is  unbounded, there is 
a sequence $\{y_j\}_{j=1}^\infty$ such that $|y_j| \to \infty$ as $j \to \infty$.
We then  get (using the chain
$x_0=z_1,x_1=y_j,x_2=z_2$) that
for sufficiently large $j$,
\begin{alignat*}{2}
    \dtilde(z_1,z_2) 
   &\le \rhotilde(z_1) \rhotilde(y_j)d(y_j,z_1) + \rhotilde(z_2) \rhotilde(y_j)d(y_j,z_2)  \\
   &\le  2(\rhotilde(z_1)+\rhotilde(z_2))\rhotilde(y_j) |y_j| 
\to 0,  & \quad &
\text{as } j \to \infty.
\end{alignat*}

So in order for $\dtilde$ to be a metric, we need 
$\liminf_{|x|\rightarrow\infty}|x|\rhotilde(|x|)>0$.
If this type of sphericalization of e.g.\   
$Z=\Rn$ is equipped
with the measure $d\tilde{\nu}= \tilde{\rho}^\sigma \,dx$, then
$\tilde{\nu}(\Rn)=\infty$ whenever $\sigma \le n$, and thus $\tilde{\nu}$ is not 
a doubling measure (since the sphericalized space is required to be bounded).
In conclusion,
using sums as in~\eqref{eq-def-dhat}, rather than products as in~\eqref{eq-def-dtilde},
gives much more flexibility.

If one instead uses the approach by 
Balogh--Buckley~\cite{BaloghBuckley} then  one is 
limited to pathconnected spaces with additional connectivity at $\infty$.
In this case Lemma~\ref{lem-d(x,y)-with-int} holds.
However, with such an approach,
the completion of
$(\R,\dhat)$ will add two points, corresponding to $\pm \infty$, and 
is thus \emph{not} a sphericalization in our sense.
Hence additional connectivity at $\infty$ is required,
which would severely
restrict the spaces under consideration.
\end{remark}

\section{Sphericalization and flattening}
\label{sect-both}

\emph{In the rest of the paper, 
we assume that the complete metric space $(Z,d)$ is equipped with a doubling measure $\nu$.}

\medskip

In order to treat sphericalization and flattening simultaneously it will be convenient to introduce the function
\begin{equation}     \label{eq-def-m(x)}
   m(t)=t+m_0,
\end{equation}
where $m_0\in \{0,1\}$
is the same constant that appears in the uniform perfectness condition.
For $x\in Z$, we let $m(x)=m(|x|)$.
We will use 
$m_0=0$ for flattening and $m_0=1$ for sphericalization.

We will use the metric density function
\begin{equation} \label{eq-special-rho}
 \rho(x)= \rho(|x|) = \frac{1}{m(x)\nu(B_{m(x)})^{1/\s}},
\quad \text{where $\s>0$ is fixed,}
\end{equation}
with the interpretation that $\rho(0)=\infty$ if $m_0=0$.
For $m_0=1$, we have $\rho(0)=1/\nu(B_1)^{1/\s}<\infty$.
Note that $\rho$ is decreasing and that by the doubling property of $\nu$,
\begin{equation}   \label{eq-rho(x)-comp-m(x)}
\rho(x) \simeq \rho(y) \quad \text{whenever } \frac1M m(x) \le m(y) \le M m(x)
\text{ for some $M>1$},
\end{equation}
where the comparison constants depend only on $C_\nu$,  $M$ and $\s$.
In particular, $\rho$ satisfies condition~\eqref{eq-doubl-rho}.

Equip $Z$ with the metric $\dhat$ defined as 
in~\eqref{eq-def-dhat} with the above $\rho$.
Recall that 
\[
R_\infty  = \sup_{x \in Z} |x|
\quad \text{and}  \quad
Z' = \begin{cases}
      Z, & \text{if } m_0=1, \\ 
      Z \setm \{b\}, & \text{if } m_0=0. 
    \end{cases}
\]
Note that $m(x) > 0$ for every $x \in Z'$.
Recall also that $\Zhat$ is the metric completion of $(Z',\dhat)$ and that we denote balls with respect to $(\Zhat,\dhat)$ by $\Bhat$.
We equip $(\Zhat,\dhat)$ with the measure
\begin{equation*}
d\nuhat=\rho^\s\, d\nu,
\end{equation*}
with the interpretation that $\nuhat(\{\binfty\})=0$ if $\binfty \in \Zhat$.

\begin{remark}  \label{rmk-inversions}
Inversions are also covered by our results:
When $Z$ is unbounded and we let $m_0=0$, we obtain an \emph{inversion}
which maps $b$ to infinity, while $\binfty$ becomes a finite point in $\Zhat$.
The results in Sections~\ref{sect-both}--\ref{sect-Besov-energy-both},
including Theorems~\ref{thm-doubl-m(x)} and~\ref{besov-preserved-both},
are formulated so that they cover sphericalizations  and flattenings, as well as 
inversions.

We leave it to the interested reader, to investigate what happens when one composes
an inversion with 
 an inversion, \`a la our results in Section~\ref{sect-duality}.
\end{remark}

In this section we deduce lemmas not needing uniform perfectness, 
while in the next section we make uniform perfectness a standing assumption.

\begin{lem} \label{lem-d(x,y)-m(x)}
Let $x,y \in Z'$ and $M>1$.
\begin{enumerate}
\item \label{it-y-ge-M-x}
If $m(y) \ge M m(x)$, then
 \begin{equation}   \label{eq-dhat-ge-m(x)}
   \dhat(x,y) \simge \rho(x) m(x) =
\frac{1}{\nu(B_{m(x)})^{1/\s}}
   \ge \frac{1}{\nu(B_{m(y)})^{1/\s}}.
\end{equation}
\item \label{it-y-le-M-x}
If $M^{-1}m(x) \le m(y) \le M m(x)$, then
 \[
   \dhat(x,y) \simeq \rho(x) d(x,y).
\]  
\end{enumerate}
\medskip 
In both cases, the comparison constants depend only on $C_\nu$, $M$ and $\s$.
\end{lem}

\begin{proof}
\ref{it-y-ge-M-x}
Lemma~\ref{lem-d(x,y)-with-int} and the doubling property of $\nu$ yield
\begin{align*}
\dhat(x,y)
&\geq \int_{|x|}^{|y|} \rho \,dt
= \int_{|x|}^{|y|} \frac{1}{m(t) \nu(B_{m(t)})^{1/\s}} \,dt
= \int_{m(x)}^{m(y)} \frac{1}{s \nu(B_s)^{1/\s}} \,ds
\\
&\geq \int_{m(x)}^{M m(x)} \frac{1}{s \nu(B_s)^{1/\s}} \,ds
\simeq \frac{1}{\nu(B_{m(x)})^{1/\s}},
\end{align*}
which proves~\eqref{eq-dhat-ge-m(x)}.

\ref{it-y-le-M-x}
The $\simle$ inequality 
follows directly from \eqref{eq-rho(x)-comp-m(x)}
by choosing the trivial chain $x=x_0,x_1=y$.

For the $\simge$ inequality, let $x=x_0, x_1, \ldots, x_n=y$ be a chain.
If $m(x_j)\le M m(x)$ for all $j$, then $\rho(x_j) \simge \rho(x)$ 
by \eqref{eq-rho(x)-comp-m(x)}, and hence
\[
\sum_{j=1}^n (\rho(x_j)+\rho(x_{j-1})) d(x_j,x_{j-1})
  \simge \rho(x) \sum_{j=1}^n d(x_j,x_{j-1}) \ge \rho(x) d(x,y).
\]
On the other hand, if $m(x_k)\ge Mm(x)$ for some $k$, then~\eqref{eq-dhat-ge-m(x)} 
with $y=x_k$ implies that
\[
\sum_{j=1}^n (\rho(x_j)+\rho(x_{j-1}))d(x_j,x_{j-1})
\ge \dhat(x,x_k) \simge \rho(x) m(x).
\]
Since also 
\[
d(x,y)\le |x|+|y| \le(M+1)m(x), 
\]
the claim follows after 
taking the infimum over all chains from $x$ to $y$.
\end{proof}

\begin{lem}   \label{lem-Bhat-both-dhat}
There are constants $c_0>0$ and $a_2\ge a_1>0$, 
depending only on $C_\nu$ and $\s$, such that 
if $x \in Z'$ and  $0<r\le c_0 \nu(B_{m(x)})^{-1/\s}$, then
\[
B\biggl(x,\frac{a_1 r}{\rho(x)}\biggr) \subset \Bhat(x,r) \subset B\biggl(x,\frac{a_2 r}{\rho(x)}\biggr) 
\quad \text{and} \quad
\rho(y) \simeq \rho(x) \quad \text{for all } y\in\Bhat(x,r).
\]
Moreover, 
\[
\nuhat(\Bhat(x,r)) 
\simeq \rho(x)^\s \nu \biggl( B\biggl(x,\frac{r}{\rho(x)}\biggr) \biggr).
\]
In this lemma the comparison constants depend only on $C_{\nu}$ and $\s$.
\end{lem}

\begin{proof}
First note that 
it follows from Lemma~\ref{lem-d(x,y)-m(x)}\ref{it-y-le-M-x}
with $M=2$ that there are $C_1\ge c_1>0$ such that
\begin{equation}  \label{eq-dhat-C_1-C_2}
c_1 \rho(x) d(x,y) \le \dhat(x,y) \le C_1 \rho(x) d(x,y)
\quad \text{when } \tfrac12m(x) \le m(y) \le 2m(x).
\end{equation}

Assume next that $y\in\Bhat(x,r)$.
If $m(y)\ge 2m(x)$ or $m(x)\ge 2m(y)$, then Lemma~\ref{lem-d(x,y)-m(x)}\ref{it-y-ge-M-x}
with $M=2$ implies that 
\[
\frac{1}{\nu(B_{m(x)})^{1/\s}} \simle \dhat(x,y) < r \le \frac{c_0}{\nu(B_{m(x)})^{1/\s}}.
\]
This leads to a contradiction if $c_0$ is sufficiently small.
Fix such a small constant $c_0>0$.
Then $\tfrac12m(x) \le m(y) \le 2m(x)$ and 
thus, by \eqref{eq-rho(x)-comp-m(x)},
\begin{equation} \label{eq-rhoy-rhox-in-lem}
\rho(y)\simeq \rho(x).
\end{equation}
It then follows from  \eqref{eq-dhat-C_1-C_2} that
\[
d(x,y) \le \frac{\dhat(x,y)}{c_1\rho(x)} < \frac{r}{c_1\rho(x)},
\]
showing that the second inclusion in the lemma holds when $a_2=1/c_1$.

On the other hand, if $d(x,y)< r/2c_0\rho(x)$, then 
\[
d(x,y)< \frac{r}{2c_0\rho(x)}  
\le \frac{c_0 \nu(B_{m(x)})^{-1/\s}}{2c_0\rho(x)} = \tfrac12 m(x).
\]
It thus follows that 
\begin{equation}        \label{eq-m(x)+-d(x,y)}
m(y) \le m(x) + d(x,y) < \tfrac32 m(x)
\quad \text{and}  \quad 
m(y) \ge m(x) - d(x,y) > \tfrac12 m(x).
\end{equation}
Letting $a_1= 1/{\max\{2c_0,C_1\}}$
and using \eqref{eq-dhat-C_1-C_2} we see that
\[
\dhat(x,y) \le C_1 \rho(x) d(x,y) < r
\quad \text{for } y \in B\biggl(x,\frac{a_1 r}{\rho(x)}\biggr),
\]
which proves the first part of the lemma.
The estimate for $\nuhat(\Bhat(x,r))$ then follows from 
\eqref{eq-rhoy-rhox-in-lem} together with
the doubling property of $\nu$.
\end{proof}

\section{The doubling property of the deformed measure}
\label{sect-with-uniform}

\emph{Recall the standing assumptions from the beginning of 
Sections~\ref{sect-prelim} and~\ref{sect-both}.
In this section we also
assume that $\rho$ is given by \eqref{eq-special-rho}
and that the complete metric space
$Z$ is uniformly perfect at $b$ with constant $\ka>1$ for radii $r\ge m_0$.
If $m_0=1$, we also assume that $Z$ is unbounded.
}


\begin{lem} \label{lem-Zhat-bdd}
\begin{enumerate}
\item \label{Z-a}
If $m_0=1$,   then $(Z,\dhat)$ is bounded.
\item \label{Z-b}
If $m_0=0$, then $(Z',\dhat)=(Z \setm \{b\},\dhat)$ is unbounded.
\end{enumerate}
\end{lem}

\begin{proof}
\ref{Z-a}
In this case, $Z$ is assumed to be unbounded.
Since Lemma~\ref{lem-measure-growth} implies that
 $\nu(B_t) \simge t^\alp \nu(B_1)$ for some $\al>0$ and all
$t \ge 1$, 
we have
\[
\int_1^\infty \rho\,dt 
\simle \int_1^\infty \frac{dt}{m(t)^{1+\alp/\s} \nu(B_1)^{1/\s}}
\le 
\frac{1}{\nu(B_1)^{1/\s}} \int_1^\infty \frac{dt}{t^{1+\alp/\s} } 
< \infty,
\]
and thus $(Z,\dhat)$ is bounded by Proposition~\ref{prop-bounded-onepoint}.

\ref{Z-b}
In this case, for any $0<\de<\diam Z$,
\[
    \int_0^{\de} \rho\,dt \ge 
\frac{1}{\nu(B_{\de})^{1/\s}}\int_0^{\de} \frac{dt}{t} 
= \infty,
\]
and thus $(Z',\dhat)$ is unbounded by Lemma~\ref{lem-d(x,y)-with-int} 
since $Z$ is uniformly perfect at $b$.
\end{proof}

\begin{lem} \label{lem-nuhat-m(x)}
For all\/ $0<r<\tfrac12 R_\infty$, 
\[
\nuhat(\Zhat \setm B_r) \simeq \nuhat(\{y\in Z: |y|>r\})
\simeq m(r)^{-\s}, 
\]
where the comparison constants depend only on $C_\nu$, $\ka$ and $\s$.
In particular,
\[
\nuhat(\Zhat)\simeq
  \begin{cases}
  1, & \text{if }m_0=1, \\
 \infty, & \text{if }m_0=0. 
 \end{cases}
\]
\end{lem}

\begin{proof}
To start with, assume that $r>m_0$ and let $r_j=(4 \ka)^jr$ for $j=0,1,\dots$\,. 
From \eqref{eq-rho(x)-comp-m(x)}, 
\eqref{eq-M} and the doubling property of $\nu$ 
we get that 
\begin{align*}
\nuhat(\Zhat \setm B_r)
&\simeq \sum_{j=0}^\infty \rho(r_j)^{\s} \nu( B_{4\ka r_{j}}\setm B_{r_j})
\simeq \sum_{r_j<\frac12 R_\infty} \rho(r_j)^{\s} \nu( B_{4\ka r_{j}}\setm B_{r_j})
\\
&\simeq \sum_{r_j<\frac12 R_\infty}  \frac{\nu( B_{r_{j}})}{m(r_j)^\s \nu( B_{m(r_j)})}
\simeq \sum_{r_j<\frac12 R_\infty}  m(r_j)^{-\s} \simeq m(r)^{-\s},
\end{align*}
where we have also used that 
\[
r_j \le m(r_j) = r_j+m_0 < 2r_j \quad \text{and hence} \quad \nu(B_{r_j}) \simeq \nu(B_{m(r_j)}).
\] 

For $m_0=1\ge r>0$, we then get from the above 
(because $\rho(0) = \nu(B_1)^{-1/\s}$ and $R_\infty=\infty$ in this case) 
and from the doubling property of $\nu$ that
\begin{align*}
m(2)^{-\s}
&\simeq \nuhat(\Zhat \setm B_2)
\leq \nuhat(\Zhat \setm B_r) \\
& \leq \nuhat(\Zhat)=
\nuhat(B_2) + \nuhat(\Zhat \setm B_2)
\simle 1 + m(2)^{-\s} \simeq 1,
\end{align*}
and thus $\nuhat(\Zhat \setm B_r) \simeq \nuhat(\Zhat) \simeq 1
\simeq m(r)^{-\s}$.

Finally, if $m_0=0$, then
\[
\nuhat(\Zhat)
\ge \lim_{r \to 0\limplus} \nuhat(\Zhat \setm B_r)
\simeq \lim_{r \to 0\limplus} m(r)^{-\s}
= \infty.
\qedhere
\]
\end{proof}

\begin{lem}  \label{lem-dhat-le-m(x)}
Let $x,y \in Z'$ with $|x| \le |y|$. 
Then
\[
   \dhat(x,y) \simle \frac{1}{\nu(B_{m(x)})^{1/\s}},
\]
where the comparison constant depends only on $C_{\nu}$, $\ka$ and $\s$.
\end{lem}

\begin{proof}
Assume first that $|x| \geq m_0$.
Note that $m(x)\simeq m(|x|/\ka)$ and so $\nu(B_{m(x)}) \simeq \nu(B_{m(|x|/\ka)})$ by the doubling property of $\nu$.
We then have by Lemmas~\ref{lem-sari} and~\ref{lem-measure-growth},
\begin{align}
\dhat(x,y)
&\simle \int_{|x|/\ka}^{|y|} \rho(t) \,dt \nonumber\\
&\leq \int_{|x|/\ka}^{R_\infty} \frac{dt}{m(t) \nu(B_{m(t)})^{1/\s}} \nonumber \\
&\simle \frac{1}{\nu(B_{m(|x| / \ka)})^{1/\s}} \int_{|x|/\ka}^{R_\infty} 
           \biggl(\frac{m(|x| / \ka)}{m(t)} \biggr)^{\al/\s} \frac{dt}{m(t)}  \nonumber \\
&\simeq \frac{ m(|x| / \ka)^{\alpha / \s}}{\nu(B_{m(x)})^{1/\s}} 
             \int_{|x|/\ka}^{R_\infty} 
             \frac{dt}{(m_0+t)^{1+\alpha / \s}}
             \nonumber \\
& \leq \frac{\s}{\al \nu(B_{m(x)})^{1/\s}}.
\label{eq-est-dhat-int}
\end{align}

Now let $|x| < m_0 = 1$.
If $|y| \leq \ka$, then
\begin{equation*}
\dhat(x,y)
\leq 2 \rho(0) d(x,y)
< \frac{2(\ka + 1)}{\nu(B_1)^{1/\s}}
\leq \frac{2(\ka+1) C_{\nu}^{1/\s}}{\nu(B_{m(x)})^{1/\s}}.
\end{equation*}
On the other hand if $|y| > \ka$, then by the uniform perfectness there exists $z \in Z$ such that $1 \leq |z| \leq \ka$ and therefore using~\eqref{eq-est-dhat-int} with $x$ replaced by $z$,
\begin{equation*}
\dhat(x,y)
\leq \dhat(x,z) + \dhat(z,y)
\simle \frac{1}{\nu(B_{m(x)})^{1/\s}} + \frac{1}{\nu(B_{m(z)})^{1/\s}}
\leq \frac{2}{\nu(B_{m(x)})^{1/\s}}.
\qedhere
\end{equation*}

\end{proof}

\begin{cor}   \label{cor-d(x,infty)-m(x)}
Let $M>1$.
Then there are constants $C_2\ge c_2>0$ 
\textup{(}depending only on $C_\nu$, $\ka$, $M$ and $\s$\textup{)}
such that for all $x,y \in Z'$ with $m(y)\ge Mm(x)$,
\begin{equation} \label{eq-d(x,infty)-m(x)}
\frac{c_2}{\nu(B_{m(x)})^{1/\s}}   \le  
\dhat(x,y) \le \frac{C_2}{\nu(B_{m(x)})^{1/\s}}.  
\end{equation}  

In particular, if $m_0 = 1$, then
\begin{equation} \label{eq-C_2-binfty}
\frac{c_2}{\nu(B_{m(x)})^{1/\s}}   \le  
\dhat(x,\binfty) \le \frac{C_2}{\nu(B_{m(x)})^{1/\s}}
\quad \text{for all }  x \in Z'.  
\end{equation}
\end{cor}

\begin{proof}
The first part follows directly from Lemmas~\ref{lem-d(x,y)-m(x)}\ref{it-y-ge-M-x}
and \ref{lem-dhat-le-m(x)}.
The last part then follows by letting $y \to \binfty$, which is possible since 
$(\Zhat,\dhat)$ is uniformly perfect at $\binfty$ when $m_0=1$, 
by Proposition~\ref{prop-unif-at-infty}.
\end{proof}

\begin{lem} \label{lem-d(x,y)-le-iff-m(x)}
Let $x,y \in Z'$, $M>0$ and $C_0>0$. 
If $m(x) \le M m(y)$, then 
\begin{equation}    \label{eq-dhat-le-nu}
   \dhat(x,y) \simle 
\frac{1}{\nu(B_{m(x)})^{1/\s}},
\end{equation} 
with the comparison constant depending only on $C_{\nu}$, $\ka$, $M$ and~$\s$.

Conversely, if $\dhat(x,y) \le C_0 \nu(B_{m(x)})^{-1/\s}$, then  $m(x) \simle m(y)$,
where the comparison constant depends only on $C_0$, $C_{\nu}$, $\ka$ and~$\s$.
\end{lem}

\begin{proof}
If $m(x)\le m(y)$ then \eqref{eq-dhat-le-nu} follows directly 
from Lemma~\ref{lem-dhat-le-m(x)}.
If instead $m(y) \le m(x) \le M m(y)$, then 
$\nu(B_{m(y)}) \simeq \nu(B_{m(x)})$ by 
the doubling property of $\nu$ and so
Lemma~\ref{lem-dhat-le-m(x)} with $x$ and $y$ interchanged gives
\[
\dhat(x,y) \simle \frac{1}{\nu(B_{m(y)})^{1/\s}} \simeq \frac{1}{\nu(B_{m(x)})^{1/\s}}.
\]

Conversely, assume that $\dhat(x,y) \le C_0 \nu(B_{m(x)})^{-1/\s}$.
If $m(x) \le 2m(y)$, then there is nothing to show. 
On the other hand, if $m(x) \ge 2m(y)$, then Lemma~\ref{lem-d(x,y)-m(x)}\ref{it-y-ge-M-x},
with $M=2$ and the roles of $x$ and $y$ interchanged, 
implies that
\[
\frac{1}{\nu(B_{m(y)})^{1/\s}}
\simle \dhat(x,y) \simle 
\frac{1}{\nu(B_{m(x)})^{1/\s}},
\]
which together with Lemma~\ref{lem-measure-growth} gives that $m(x) \simle m(y)$.
\end{proof}

\begin{prop} \label{prop-nuhat-comp-r}
Let $0 < c_0 \le  C_0$.
Assume that $x \in Z'$ and
\[
\frac{c_0}{\nu(B_{m(x)})^{1/\s}} \le r \le \frac{C_0}{\nu(B_{m(x)})^{1/\s}}.
\]
Then
\begin{equation} \label{eq-nuhat-Bhat}
\nuhat(\Bhat(x,r)) \simeq m(x)^{-\s},
\end{equation}
with comparison constants depending 
only on $c_0$, $C_0$, $C_{\nu}$, $\ka$ and $\s$.
\end{prop}

Proposition~\ref{prop-nuhat-comp-r} holds for any $c_0>0$,
but will  sometimes be used with
the constant $c_0$ given by Lemma~\ref{lem-Bhat-both-dhat}.

\begin{proof}
Let $y \in \Bhat(x,r)$.
Then 
\[
\dhat(x,y)<r \le \frac{C_0}{\nu(B_{m(x)})^{1/\s}}
\]
and Lemma~\ref{lem-d(x,y)-le-iff-m(x)} implies that there is $M>2$ 
such that $m(x) \le M m(y)$.

Let us first consider the case $m(x)>Mm_0$. 
Then $|y| \ge m(x)/M-m_0 =: r_0>0$ and therefore $\Bhat(x,r)\subset \Zhat \setm B_{r_0}$.
As $M>2$ and $m(x)\le R_\infty+m_0$, we see that now $0<r_0<\tfrac12 R_\infty$.
Lemma~\ref{lem-nuhat-m(x)} gives that
\begin{equation*}
\nuhat(\Bhat(x,r)) \le \nuhat(\Zhat \setm B_{r_0})
\simeq m(r_0)^{-\s} =  \biggl( \frac{m(x)}{M} \biggr)^{-\s}.
\end{equation*}

On the other hand, if $m(x)\le Mm_0$, then $m_0=1$, because $m(x) > 0$ for every $x \in Z'$. 
Therefore, by Lemma~\ref{lem-nuhat-m(x)},
\[
\nuhat(\Bhat(x,r)) \le \nuhat(\Zhat) 
\simeq 1 \simeq m(x)^{-\s},
\]
which completes the proof of the $\simle$ inequality in~\eqref{eq-nuhat-Bhat}.

For the $\simge$ inequality,
Lemma~\ref{lem-d(x,y)-m(x)}\ref{it-y-le-M-x} with $M=2$ 
implies that there is $C_1$ such that
\begin{equation}  \label{eq-dhat-C_2}
\dhat(x,y) \le C_1 \rho(x) d(x,y) 
\quad \text{whenever} \quad
\tfrac12 m(x) \le m(y) \le 2m(x).
\end{equation}
Let $a_0 := \min \{ c_0/C_1, \tfrac{1}{2}\}$ and consider points $y$ with
 $d(x,y)<a_0m(x)$.
Then as in \eqref{eq-m(x)+-d(x,y)} we have
\[
\tfrac{1}{2} m(x) \leq m(y) \leq \tfrac{3}{2} m(x)
\]
and therefore \eqref{eq-dhat-C_2}, 
the definition of $\rho$ and the assumption of this proposition imply that
\[
\dhat(x,y) \leq C_1 \rho(x) d(x,y) < C_1 \rho(x) a_0 m(x) 
\leq \frac{c_0}{\nu(B_{m(x)})^{1/\s}} \leq r.
\]
Thus $B(x,a_0m(x)) \subset \Bhat(x,r) $.
Since also $\rho(y) \simeq \rho(x)$ by \eqref{eq-rho(x)-comp-m(x)},
we get from Lemma~\ref{lem-comp-close-balls} and the doubling property of $\nu$ that
\begin{align*}
\nuhat(\Bhat(x,r))
&\geq \nuhat(B(x,a_0m(x))) \\
&\simeq \rho(x)^\s \nu(B(x,a_0m(x)))
\simeq \rho(x)^\s \nu(B_{m(x)})
= m(x)^{-\s}.
\qedhere
\end{align*}
\end{proof}

To prove that the doubling property of the measure is preserved under our transformations, it is convenient to introduce the following ``inverse'' of $\nu$,
\begin{equation}    \label{eq-def-nuinv}
\nuinv(t) = \sup \{r \geq 0 : \nu(B_r) \le t\} = \inf \{r > 0 : \nu(B_r) > t\}
\end{equation}
for $0 \le t<\nu(Z)$.
It follows immediately from the regularity of $\nu$ that
\begin{equation}
\label{eq-nu-nuinv}
\nu(B_{\nu^{-1}(t)}) = \sup_{0\le r < \nu^{-1}(t)} \nu(B_r)
\le t 
\quad \text{for all  } 0 \leq t < \nu(Z),
\end{equation}
and consequently,
\begin{equation}
\label{eq-nuinv-iff}
r > \nuinv(t) 
\quad \text{if and only if} \quad  
\nu(B_r)  > t.
\end{equation}
In particular, the supremum in~\eqref{eq-def-nuinv} is always attained.

\begin{lem} \label{lem-nuinv-2-new}
Let $M \geq 1$.
Then for all $0 \leq t_1 \leq t_2 \leq M t_1$, such that $\nu(B_{m_0}) \le t_1 \leq t_2 < \nu(Z)$,
\begin{equation}
\label{eq-comp-nuinv-2t}
\nuinv(t_2)
\simle \nuinv(t_1)
\le \nuinv(t_2),
\end{equation}
with the comparison constant depending only on $C_{\nu}$, $\ka$ and $M$.
In particular if $\nu(B_{m_0}) \le t < \tfrac{1}{2} \nu(Z)$, then $\nuinv(2t) \simeq \nuinv(t)$.
\end{lem}

\begin{proof}
The second inequality in~\eqref{eq-comp-nuinv-2t} is immediate, because $\nuinv$ is nondecreasing.
For the first inequality, let $r_1 = \nu^{-1}(t_1)$ and $r_2 = \nu^{-1}(t_2)$.
If $r_1=r_2$, there is nothing to prove, so assume that $r_1 < r_2$.

Let $\eps>0$ be such that $r_1+\eps<r_2$.
By \eqref{eq-nu-nuinv} we have $\nu(B_{r_2}) \le t_2 < \nu(Z)$ and thus $r_2 \le R_\infty$.
Moreover by the definition of $\nuinv(t_1)$
we have $r_1 \ge m_0$ and by \eqref{eq-nuinv-iff} we get $\nu(B_{r_1+\eps})>t_1$.
It thus follows from Lemma~\ref{lem-measure-growth} (with $\Lambda$ and $\alp$ as therein) that
\[
    \frac1M \leq \frac{t_1}{t_2} < \frac{\nu(B_{r_1+\eps})}{\nu(B_{r_2})} \le 
    \La \Bigl(\frac{r_1+\eps}{r_2}\Bigr)^\alp.
\]
Letting $\eps \to 0$ shows that
\begin{equation*}
\nuinv(t_2) = r_2 \le (\La M)^{1/\al} r_1 =
(\La M)^{1 / \alp} \nuinv(t_1).
\qedhere
\end{equation*}
\end{proof}

\begin{prop}
\label{prop-nuhat-large-r}
There is a constant $C'>1$ depending only on $C_{\nu}$, $\ka$ and $\s$ such that for all $x \in Z'$ and 
$r>C' \nu(B_{m(x)})^{-1/\s}$,
\begin{equation} \label{eq-nuhat-large-r}
\nuhat(\Bhat(x,r)) \simeq \begin{cases} 
     \nuinv(r^{-\s})^{-\s}, & \text{if } r\le c_2 \nu(B_{m_0})^{-1/\s}, \\
     \nuhat(Z), & \text{if } r\ge \tfrac12 c_2 \nu(B_{m_0})^{-1/\s}.
     \end{cases}
\end{equation}
with comparison constants depending only on $C_{\nu}$, $\ka$ and $\s$.
Here $c_2\le1$ is a constant from Corollary~\ref{cor-d(x,infty)-m(x)} with $M=2$.
\end{prop}

Note that $\nu(B_{m_0}) = \nu(\emptyset) = 0$ if $m_0=0$, in which case we 
use the interpretation
$\nu(B_{m_0})^{-1/\s} = \infty$. 
Thus the second case in~\eqref{eq-nuhat-large-r} is possible only if $m_0=1$.

\begin{proof}
Let $C_2 \ge c_2 >0$ be the constants
from Corollary~\ref{cor-d(x,infty)-m(x)} with $M=2$.
We can assume that $c_2\le1\le C_2$.
Define $C'=\max\{C_2 C_\nu^{1/\s},3C_1\}>1$, where $C_1$ is as in~\eqref{eq-dhat-C_2}.
We shall show that 
\begin{equation}    \label{eq-bound-Bhat-r1-r2}
\{ y\in Z : \nu(B_{m(y)}) > t_2\}
\subset \Bhat(x,r) \cap Z \subset  
\{ y\in Z : \nu(B_{m(y)}) > t_1\},
\end{equation}
where
\[
t_1 := \Bigl( \frac{c_2}{r} \Bigr)^\s 
\le t_2 := \biggl( \frac{C_2}{r} \biggr)^\s.
\]
To this end, let $y\in Z'$.
If $m(y)\ge2m(x)$, then by Corollary~\ref{cor-d(x,infty)-m(x)} 
and the assumption $r>C' \nu(B_{m(x)})^{-1/\s}$, we get that
\[
\dhat(x,y) \le \frac{C_2}{\nu(B_{m(x)})^{1/\s}} <r
\]
and so $y\in\Bhat(x,r)$. Similarly, if $\tfrac12 m(x) \le m(y) \le 2m(x)$, then
by \eqref{eq-dhat-C_2},
\[
\dhat(x,y) \le C_1\rho(x) d(x,y) \le \frac{C_1( m(x)+m(y))}{m(x) \nu(B_{m(x)})^{1/\s}}
\le \frac{3C_1}{\nu(B_{m(x)})^{1/\s}}<r
\]
and so $y\in\Bhat(x,r)$ also in this case.
Finally, if $2m(y) \le m(x)$ and $\nu(B_{m(y)}) > t_2= (C_2/r)^\s$, then
Corollary~\ref{cor-d(x,infty)-m(x)} with $x$ and $y$ interchanged gives
\[
\dhat(x,y) \le \frac{C_2}{\nu(B_{m(y)})^{1/\s}} < r, 
\quad \text{i.e.\ } y\in \Bhat(x,r).
\]
This proves the first inclusion in \eqref{eq-bound-Bhat-r1-r2}.

In order to prove the second inclusion in \eqref{eq-bound-Bhat-r1-r2}, 
let $y\in \Bhat(x,r)\cap Z$.
If $2m(y) \le m(x)$, then Corollary~\ref{cor-d(x,infty)-m(x)} with $M=2$ and the roles of $x$ and $y$ interchanged shows that
\[
r > \dhat(x,y) \ge \frac{c_2}{\nu(B_{m(y)})^{1/\s}}
\quad \text{and so} \quad
\nu(B_{m(y)}) > \Bigl( \frac{c_2}{r} \Bigr)^\s = t_1.
\]
On the other hand, if $2m(y)\ge m(x)$, then
the doubling property of $\nu$
and the assumption $r > C' \nu(B_{m(x)})^{-1/\s}$ yield
\[
C_\nu \nu(B_{m(y)}) \ge \nu(B_{m(x)}) > \biggl( \frac{C'}{r} \biggr)^\s
\ge C_\nu \biggl( \frac{C_2}{r} \biggr)^\s
\ge C_\nu \Bigl( \frac{c_2}{r} \Bigr)^\s,
\]
proving the second inclusion in \eqref{eq-bound-Bhat-r1-r2}.

Next we use Lemma~\ref{lem-nuhat-m(x)} to estimate the $\nuhat$-measure 
of the sets in \eqref{eq-bound-Bhat-r1-r2}.
For this we note that when $R_\infty<\infty$ (and thus $m_0=0$), 
we have by the doubling property of $\nu$ that
\[
t_1\le t_2 < \biggl( \frac{C_2}{C'} \biggr)^\s \nu(B_{m(x)})
\le \frac{\nu(B_{R_\infty})}{C_\nu}  \le \nu(B_{R_\infty/2})< \nu(Z).
\]
It then follows from~\eqref{eq-nuinv-iff}  that 
\begin{equation*}  
r_j:=\nuinv(t_j) < \tfrac12 R_\infty, \quad j=1,2,
\end{equation*}
while this is trivial when $R_\infty=\infty$.

We now distinguish two cases.
If $r \le c_2 \nu(B_{m_0})^{-1/\s}$, then $t_1\ge \nu(B_{m_0})$ 
and thus $r_1\ge m_0$, by \eqref{eq-nuinv-iff}.
The inclusions~\eqref{eq-bound-Bhat-r1-r2}, together with \eqref{eq-nuinv-iff} 
and Lemma~\ref{lem-nuhat-m(x)}, then imply that
\begin{equation*}   
\nuhat(\Bhat(x,r)) \le
\nuhat (\{y\in Z: m(y)>r_1\}) \le
\nuhat (Z\setm B_{r_1-m_0}) 
\simeq m(r_1-m_0)^{-\s} = r_1^{-\s},
\end{equation*}
and similarly
\begin{equation*}   
\nuhat(\Bhat(x,r)) \ge \nuhat (\{y\in Z: m(y)>r_2\}) 
\simeq m(r_2-m_0)^{-\s} = r_2^{-\s}.
\end{equation*}
Since $t_1\le r^{-\s} \le t_2$,
Lemma~\ref{lem-nuinv-2-new} shows that $r_1\simeq r_2 \simeq \nuinv(r^{-\s})$,
which proves the first estimate in \eqref{eq-nuhat-large-r}
for such $r$.

Next, if $r \ge \tfrac12 c_2 \nu(B_{m_0})^{-1/\s}$, then $m_0=1$ 
(and thus $(Z,d)$ is unbounded by assumption) and $t_2\le (2C_2/c_2)^\s\nu(B_{m_0})=: t_0$.
Hence $\nuinv(t_0)\ge m_0$ and by \eqref{eq-bound-Bhat-r1-r2},  \eqref{eq-nuinv-iff} 
and Lemma~\ref{lem-nuhat-m(x)},
\begin{align*}   
\nuhat (Z) &\ge
\nuhat(\Bhat(x,r)) \ge \nuhat (\{y\in Z: \nu(B_{m(y)}) > t_0\}) \\
&= \nuhat (\{y\in Z: m(y)>\nuinv(t_0)\})
\simeq m(\nuinv(t_0)-m_0)^{-\s} 
= \nuinv(t_0)^{-\s}.
\end{align*}
Since also $t_0 \simle \nu(B_{m_0})$, we conclude using
\eqref{eq-nuinv-iff} and Lemma~\ref{lem-measure-growth} that
$\nuinv(t_0) \simle m_0 = 1$, 
which together with $\nuhat(Z) \simeq 1$ from Lemma~\ref{lem-nuhat-m(x)}
implies the second estimate in \eqref{eq-nuhat-large-r}.
\end{proof}

\begin{thm}   \label{thm-doubl-m(x)}
The measure $\nuhat$ is doubling on $(Z',\dhat)$ and on its completion $(\Zhat,\dhat)$. 
The doubling constant $C_{\nuhat}$ depends only on $C_{\nu}$, $\ka$ and $\s$.
\end{thm}

\begin{proof}
Let $x \in Z'$.
If $0 < r \leq \tfrac12 c_0 \nu(B_{m(x)})^{-1/\s}$, where $c_0$ is the constant 
in Lemma~\ref{lem-Bhat-both-dhat}, then we get using 
Lemma~\ref{lem-Bhat-both-dhat} and
the doubling property of $\nu$ that
\[
\nuhat(\Bhat(x,2r)) \simeq \rho(x)^\s \nu \biggl( B\biggl(x,\frac{r}{\rho(x)}\biggr) \biggr)
\simeq \nuhat(\Bhat(x,r)).
\]
If $r>C' \nu(B_{m(x)})^{-1/\s}$, where $C'>1$ is as in 
Proposition~\ref{prop-nuhat-large-r},
then
\[
r^{-\s} < (C')^{-\s} \nu(B_{m(x)}) < \nu(Z).
\]
Hence,
by Proposition~\ref{prop-nuhat-large-r} and Lemma~\ref{lem-nuinv-2-new},
either
\[
\nuhat(\Bhat(x,2r)) 
\simeq \nuinv((2r)^{-\s})^{-\s}
\simeq \nuinv(r^{-\s})^{-\s}
\simeq \nuhat(\Bhat(x,r))
\]
(when $2r\le c_2\nu(B_{m_0})^{-1/\s}$)
or $\nuhat(\Bhat(x,2r)) \simeq \nuhat(Z)\simeq \nuhat(\Bhat(x,r))$.

Finally if $\tfrac12 c_0 \nu(B_{m(x)})^{-1/\s} \leq r \leq C' \nu(B_{m(x)})^{-1/\s}$, 
then Proposition~\ref{prop-nuhat-comp-r} yields that
\begin{equation*}
\nuhat(\Bhat(x,2r)) \simeq m(x)^{-\s}
\simeq \nuhat(\Bhat(x,r)).
\end{equation*}
Thus $\nuhat$ is doubling on $(Z',\dhat)$.
The same doubling constant can be used also on $(\Zhat,\dhat)$,
by Proposition~3.3 in Bj\"orn--Bj\"orn~\cite{BBnoncomp}.
\end{proof}

\section{Preservation of the Besov energy}
\label{sect-Besov-energy-both}

\emph{Recall the standing assumptions from the beginning of 
Sections~\ref{sect-prelim} and~\ref{sect-both}.}

\medskip

The aim of this section is to show that the transformation of the metric 
and the measure introduced in Section~\ref{sect-both} preserves 
the Besov (fractional Sobolev) energy~\eqref{eq-Besov-seminorm} 
if we choose $\s = p \theta$.
Recall that the metric $\dhat$ was defined in~\eqref{eq-def-dhat} 
and~\eqref{eq-def-m(x)}--\eqref{eq-special-rho} with 
\begin{equation}    \label{eq-recall-rho-m(x)}
\rho(x)= \frac{1}{m(x) \nu(B_{m(x)})^{1/\s}}
\quad \text{and} \quad  m(x)=|x|+m_0,
\end{equation}
and that 
\[
d\nuhat = \rho^\s\,d\nu
\]
with $\nuhat( \{ \binfty\} ) = 0$  if $\binfty \in \Zhat$.

\begin{thm}
\label{besov-preserved-both}
Assume that $Z$ is uniformly perfect at $b$ for radii $r\ge m_0$, 
with constant $\ka>1$, and equipped with a doubling measure $\nu$.
Let $1\le p<\infty$, $\theta>0$ and $\s=p\theta$.

Then for every measurable function $u$ on $Z$, the 
Besov energies~\eqref{eq-Besov-seminorm} 
with respect to $(Z,d,\nu)$ 
and $(\Zhat,\dhat,\nuhat)$ are comparable, 
with comparison constants depending only on $C_\nu$, $\ka$ and $\s$.
\end{thm}

\begin{proof}
It suffices to consider what happens to the part of the integrand in the 
Besov 
energy~\eqref{eq-Besov-seminorm} that depends on the metric or on the measure when we transform them.
More precisely, the comparability of the 
Besov energy after the transformation 
will follow if we can show that for every $x,y\in Z'$ with $x \neq y$,
\begin{equation}   \label{eq-prove-comp}
\frac{\rho(x)^\s\rho(y)^\s}{\dhat(x,y)^{\s}\nuhat(\Bhat_{xy})}
\simeq 
\frac{1}{d(x,y)^{\s}\nu(B_{xy})},
\end{equation}
where
\[
B_{xy} := B(x,d(x,y)) \quad \text{and}  \quad \Bhat_{xy} := \Bhat(x,\dhat(x,y)). 
\]
(It suffices to consider $x,y\in Z'$, because $\nuhat(\{\binfty\}) = 0$ 
(if $\binfty \in \Zhat$),
and if $Z' \neq Z$ then $m_0=0$ and therefore $\nu(\{b\}) = 0$.)
By the doubling properties of $\nu$ and $\nuhat$, see Theorem~\ref{thm-doubl-m(x)},
together with Lemma~\ref{lem-comp-close-balls}, we have 
\[
\nu(B_{xy}) \simeq \nu(B_{yx})   \quad \text{and}  \quad 
\nuhat(\Bhat_{xy}) \simeq \nuhat(\Bhat_{yx}). 
\]
Therefore by symmetry, it is enough to consider points with 
$|x| \le |y|$.
We now consider several cases.

{\bf Case 1:}  $m(x) \le m(y) \le 2m(x)$. 
Then by \eqref{eq-rho(x)-comp-m(x)} and Lemma~\ref{lem-d(x,y)-m(x)}\ref{it-y-le-M-x}
with $M=2$,
\begin{equation}    \label{eq-comp-d-dhat}
\rho(y) \simeq \rho(x) 
\quad \text{and} \quad 
\dhat(x,y) \simeq \rho(x) d(x,y).
\end{equation}
Hence, in order to show \eqref{eq-prove-comp}, it is enough to show that
\begin{equation}   \label{eq-prove-comp-2}
\frac{\rho(x)^\s}{\nuhat(\Bhat_{xy})} 
\simeq 
\frac{1}{\nu(B_{xy})}.
\end{equation}

{\bf Case 1a:}
$\dhat(x,y)\le c_0 \nu(B_{m(x)})^{-1/\s} $,
where $c_0>0$ is as in Lemma~\ref{lem-Bhat-both-dhat}.
Then
Lemma~\ref{lem-Bhat-both-dhat}, the doubling property of $\nu$ and \eqref{eq-comp-d-dhat}
imply that
\[
\nuhat(\Bhat_{xy}) 
\simeq \rho(x)^\s \nu \biggl( B\biggl(x,\frac{\dhat(x,y)}{\rho(x)}\biggr) \biggr)
\simeq \rho(x)^\s \nu(B_{xy}),
\]
i.e.\ \eqref{eq-prove-comp-2} holds.

{\bf Case 1b:}
$\dhat(x,y)\ge c_0 \nu(B_{m(x)})^{-1/\s} $.
In this case, we have by \eqref{eq-recall-rho-m(x)} and \eqref{eq-comp-d-dhat} that
\[
d(x,y) \simeq \frac{\dhat(x,y)}{\rho(x)}  
= m(x)\nu(B_{m(x)})^{1/\s} \dhat(x,y)
\ge c_0 m(x).
\]
Hence also
\[
m(x) \simle d(x,y) \le m(x) + m(y) \le 3m(x)
\]
and consequently,
\[
\dhat(x,y) \simeq \rho(x) d(x,y) \simeq \rho(x) m(x) =
\frac{1}{\nu(B_{m(x)})^{1/\s}}.
\]
Proposition~\ref{prop-nuhat-comp-r} then shows that
\[
\nuhat(\Bhat_{xy}) \simeq m(x)^{-\s},
\]
while Lemma~\ref{lem-comp-close-balls}
and~\eqref{eq-recall-rho-m(x)} 
imply that
\[
\nu(B_{xy}) \simeq \nu(B_{m(x)}) = (\rho(x) m(x))^{-\s}
\simeq \frac{\nuhat(\Bhat_{xy})}{\rho(x)^\s},
\]
i.e.\ \eqref{eq-prove-comp-2} holds.

{\bf Case 2:}  $m(y) \ge 2 m(x)$. 
In this case we have by Corollary~\ref{cor-d(x,infty)-m(x)} that
\[
\dhat(x,y) \simeq 
\frac{1}{\nu(B_{m(x)})^{1/\s}}= \rho(x) m(x)
\]
and hence by 
Proposition~\ref{prop-nuhat-comp-r},
\[
\nuhat(\Bhat_{xy}) \simeq m(x)^{-\s}.
\]
Since also 
\begin{equation}   \label{eq-1/2-d(x,y)-3/2}
d(x,y) \ge m(y) -m(x) \ge \tfrac12 m(y) 
\quad \text{and} \quad
d(x,y) \le m(y)+m(x) \le \tfrac32 m(y),
\end{equation}
the doubling property of $\nu$ and Lemma~\ref{lem-comp-close-balls} imply that 
\[
\nu(B_{xy}) \simeq \nu(B_{yx}) \simeq \nu(B_{m(y)}).
\]
It therefore follows that
\[
\frac{\rho(x)^\s\rho(y)^\s}{\dhat(x,y)^{\s}\nuhat(\Bhat_{xy})}
\simeq \frac{\rho(x)^\s\rho(y)^\s}{(\rho(x) m(x))^\s m(x)^{-\s}}
= \frac{1}{m(y)^\s \nu(B_{m(y)})}
\simeq \frac{1}{d(x,y)^{\s}\nu(B_{xy})}.\qedhere
\]
\end{proof}

\section{Duality of sphericalization and flattening} 
\label{sect-duality}

In this section, we show that flattening and sphericalization are inverse processes.
In Section~\ref{sect-flatt-sp} we consider the flattening of a sphericalized space, while in Section~\ref{sect-sp-flatt} we consider the reverse order.

Note that in this section we do \emph{not} require $\s = p \theta$ (as we did in Section~\ref{sect-Besov-energy-both}).

\subsection{The flattening of a sphericalized space}
\label{sect-flatt-sp}

\emph{In this section we assume that $(Z,d,\nu)$ is an unbounded complete metric space  
equipped with a doubling measure $\nu$ with a doubling constant $C_{\nu}$. 
We also assume that $Z$ is uniformly perfect at large scales at a base point $b$ with constant $\ka > 1$. 
}

\medskip

We denote the sphericalization of the space $(Z,d,\nu)$ by $(\Zhat, \dhat, \nuhat)$ as defined in Section~\ref{sect-both} with the function $m(x)=|x|+1$ and $\Zhat = Z \cup \{ \binfty \}$.
By Proposition~\ref{prop-unif-at-infty} and Lemma~\ref{lem-Zhat-bdd}, $(\Zhat,\dhat)$ is bounded and uniformly perfect at the point $\binfty$.
Note that the doubling constant in~\eqref{eq-doubl-rho} can be chosen so that $A\le 2C_\nu^{1/\s}$.
Moreover, $\rho(0)/\rho(1) \le 2 C_\nu^{1/\s}$ and thus 
the uniform perfectness constant $\kahat$ of $(\Zhat,\dhat)$ 
(from Proposition~\ref{prop-unif-at-infty}) only depends on $C_{\nu}$, $\ka$ and $\s$.
By Theorem~\ref{thm-doubl-m(x)} the measure $\nuhat$ is doubling in $(\Zhat,\dhat)$ 
with a doubling constant depending only on $C_{\nu}$, $\ka$ and $\s$.

Now let $(Z, \dtilde, \nutilde)$ be the flattening of the sphericalization 
$(\Zhat,\dhat,\nuhat)$ of $(Z, d, \nu)$.
The flattening is done at the base point $\binfty$ as in Section~\ref{sect-both} 
by using the gauge function $\mhat(x)=\dhat(x,\binfty)$.
Observe that the flattening removes $\binfty$ and thus gives the space $Z$ back.
Let us denote by $\rho$ the metric density function used in the sphericalization and by $\rhohat$ the metric density function used in the flattening, i.e.
\begin{equation} \label{eq-special-rho-flat-sp}
 \rho(x)=  \frac{1}{m(x)\nu(B_{m(x)})^{1/\s}}
\quad \text{and} \quad
 \rhohat(x)
= \frac{1}{\dhat(x,\binfty) \nuhat (\Bhat(\binfty,\dhat(x,\binfty)))^{1/\s}}.
\end{equation}

\begin{lem} \label{lem-rho-rhohat-new}
Let $x \in Z$. Then $\rho(x) \rhohat(x) \simeq 1$. 
In particular, $d \nutilde (x) \simeq d \nu (x)$. 
Here, the comparison constants depend only on $C_\nu$, $\ka$ and $\s$.
\end{lem}

\begin{proof}
We have by \eqref{eq-special-rho-flat-sp},
together with Corollary~\ref{cor-d(x,infty)-m(x)},
the doubling property of $\nuhat$ 
and Lemma~\ref{lem-comp-close-balls} (for $\nuhat$), that
\begin{align}   
\frac{1}{\rho(x) \rhohat(x)}
&= m(x) \nu(B_{m(x)})^{1/\s} \dhat(x,\binfty) \nuhat (\Bhat(\binfty,\dhat(x,\binfty)))^{1/\s} 
\nonumber \\
&\simeq m(x) \nuhat \biggl( \Bhat\biggl(x,\frac{c_0}{\nu(B_{m(x)})^{1/\s}}\biggr) \biggr)^{1/\s},
\label{eq-rho-rhohat}
\end{align}
where $c_0$ is as in Lemma~\ref{lem-Bhat-both-dhat}.
Since 
\[
\frac{c_0}{\nu(B_{m(x)})^{1/\s}} = c_0 m(x)\rho(x),
\]
Lemmas~\ref{lem-Bhat-both-dhat} and
\ref{lem-comp-close-balls}  yield
\[
\nuhat \biggl( \Bhat\biggl(x,\frac{c_0}{\nu(B_{m(x)})^{1/\s}}\biggr) \biggr)
\simeq \rho(x)^\s \nu(B(x,c_0 m(x)))
\simeq \rho(x)^\s \nu(B_{m(x)}) =
m(x)^{-\s}.
\]

Inserting this into \eqref{eq-rho-rhohat} concludes the proof.
\end{proof}

\begin{thm}   \label{thm-flatt-spher}
The flattening of the sphericalized space is equivalent to the original space in the following sense:
The metrics $\dtilde$ and $d$ are equivalent and the measures $\nutilde$ and $\nu$ are comparable, that is, for all $x,y \in Z$,
\[
 \dtilde (x,y) \simeq d(x,y) \quad \text{and}   \quad d\nutilde (x) \simeq d\nu (x),
\]
where the comparison constants depend only on 
$C_\nu$, $\ka$ and $\s$.
\end{thm}

\begin{proof} 
The comparability of the measures was already proved in Lemma~\ref{lem-rho-rhohat-new}.
For the metrics we may assume by symmetry that $\dhat(y,\binfty) \le \dhat(x,\binfty)$.

\textbf{Case 1:} $\dhat(x,\binfty) \ge 2 \dhat(y,\binfty)$.
In this case, Corollary~\ref{cor-d(x,infty)-m(x)}
(first \eqref{eq-d(x,infty)-m(x)} applied to $\dhat$ and then 
\eqref{eq-C_2-binfty} applied to $d$),
together with
Lemma~\ref{lem-comp-close-balls}, yields
\[
\dtilde(x,y)   \simeq \nuhat(\Bhat(\binfty,\dhat(y,\binfty)))^{-1/\s}
\simeq \nuhat \biggl( \Bhat\biggl(y,\frac{1}{\nu(B_{m(y)})^{1/\s}}\biggr) \biggr)^{-1/\s}
\simeq m(y),
\]
where the last comparison follows from Proposition~\ref{prop-nuhat-comp-r}
with $r=\nu(B_{m(y)})^{-1/\s}$. 

If $m(y)\ge 2m(x)$ then as in~\eqref{eq-1/2-d(x,y)-3/2},
also $\tfrac12 m(y) \le d(x,y) \le \tfrac32 m(y)$ and the result follows.
On the other hand, if $m(y)\le 2m(x)$ then
$\nu(B_{m(y)}) \simle \nu(B_{m(x)})$ by the doubling 
property of $\nu$,
and thus 
Corollary~\ref{cor-d(x,infty)-m(x)} implies that we are in the following case:

\textbf{Case 2:} $\dhat(x,\binfty) \simle \dhat(y,\binfty)\le \dhat(x,\binfty)$. 
In particular, $\mhat(x)\simeq\mhat(y)$.
It also follows from Corollary~\ref{cor-d(x,infty)-m(x)} that 
\(
\nu(B_{m(x)}) \simeq \nu(B_{m(y)})
\),
and thus $m(x) \simeq m(y)$ by Lemma~\ref{lem-measure-growth}.
Using also Lemma~\ref{lem-d(x,y)-m(x)}\ref{it-y-le-M-x} (applied first to $\dtilde$ and then to $\dhat$) and Lemma~\ref{lem-rho-rhohat-new} we get that
\[
\dtilde(x,y) \simeq  \rhohat(x) \dhat(x,y) \simeq \rhohat(x) \rho(x) d(x,y) 
\simeq d(x,y).\qedhere
\]
\end{proof}

\subsection{The sphericalization of a flattened space}
\label{sect-sp-flatt}

\emph{In this section we assume that $(Z, d)$ is a bounded complete metric space 
equipped  with a doubling measure $\nu$ with a doubling constant $C_{\nu}$. 
We also assume that $Z$ is uniformly perfect at a base point $b$ with constant $\ka>1$. 
}

\medskip

We denote the flattening of this space by $(Z', \dhat, \nuhat)$ as defined 
in Section~\ref{sect-both} with the gauge function $m(x)=|x|$
and $Z'=Z\setm\{b\}$.
By Lemma~\ref{lem-Zhat-bdd}\ref{Z-b} the space $(Z',\dhat)$ is unbounded, 
and by Proposition~\ref{prop-unbounded-complete} it  
is complete, so $\Zhat = Z'$. 
Theorem~\ref{thm-doubl-m(x)} shows that the measure $\nuhat$ is 
doubling on $(\Zhat,\dhat)$ with a doubling constant 
depending only on $C_\nu$, $\ka$ and $\s$.

Note that $\rho$ is doubling and the 
doubling constant $A$ for $\rho$
in~\eqref{eq-doubl-rho} can be chosen so that $A\le 2C_\nu^{1/\s}$. 
Therefore by Proposition~\ref{prop-unif-at-b'}, the space $(\Zhat,\dhat)$ 
is uniformly perfect at large scales at 
a base point $\bhat \in \Zhat$
with a constant $\kahat > 1$ that depends only on 
$C_{\nu}$, $\ka$, $\s$ and $|\bhat| \rho(\bhat) = \nu(B_{\bhat})^{-1/\s}$, where we let 
\[
B_z=B(b,|z|) 
\quad \text{for } z \in Z
\]
in this section.

Then we sphericalize the space $(\Zhat, \dhat, \nuhat)$ 
at the base point $\bhat$, as defined in Section~\ref{sect-both} 
with the function $\mhat(x)=\dhat(x,\bhat)+1$.
After sphericalization, we get a new space, which we denote by $(\Zt, \dtilde, \nutilde)$.
The act of sphericalization adds one point to $\Zhat$, which is denoted $\binfty$. 
It follows from Case~3 in the proof of Theorem~\ref{thm-spher-flatt} below 
that 
we can
identify $\binfty$ with the point $b$ that was removed in the flattening.
Thus $\Zt= (Z \setminus \{b\} )\cup \{\binfty\} = Z$.

Denote by $\rho$ the metric density function used in the flattening and by $\rhohat$ the metric density function used in the sphericalization, i.e.\
\begin{equation*} 
 \rho(x)=  \frac{1}{|x| \nu(B_x)^{1/\s}}
\quad \text{and} \quad
 \rhohat(x)
= \frac{1}{\mhat(x) \nuhat(\Bhat(\bhat,\mhat(x)))^{1/\s}}.
\end{equation*}
Recall that 
\[
    R_\infty := \sup_{x \in Z} |x| <\infty.
\]

\begin{lem} \label{lem-rho-rhohat-2-new}
Let $x \in Z\setm\{b\}$. 
Then $\rho(x) \rhohat(x) \simeq 1$.
In particular, we have $d \nutilde(x) \simeq d \nu(x)$ for every 
$x \in Z\setm\{b\}$.
The comparison constants depend only on $C_{\nu}$, $\ka$, $\s$, $R_\infty$, $\nu(Z)$ and $|\bhat|$.
\end{lem}

\begin{proof}
In this case we have
\begin{equation}   \label{eq-prod-rho-rhohat}
\frac{1}{\rho(x) \rhohat(x)}
= |x| \nu(B_x)^{1/\s} \mhat(x) \nuhat(\Bhat(\bhat,\mhat(x)))^{1/\s},
\end{equation}
where $\mhat(x)=\dhat(x,\bhat)+1$.
Note that $|\bhat|\simeq 1$ 
and $\nu(B_{\bhat})\simeq 1$
(since the comparison constants are allowed to depend on 
$C_{\nu}$, $R_\infty$, $\nu(Z)$ and $|\bhat|$).

If $|x|\ge \tfrac12 |\bhat|$ then also $|x|\simeq 1$. 
Moreover, Lemma~\ref{lem-d(x,y)-m(x)}\ref{it-y-le-M-x} implies that
\[
\dhat(x,\bhat) \simeq \rho(\bhat) d(x,\bhat) \simeq d(x,\bhat) 
\le d(x,b) + d(b,\bhat)
\le 2 R_\infty.
\]
Hence 
\[
\mhat(x) = \dhat(x,\bhat)+1 \simeq 1
\quad \text{and consequently} \quad
\nuhat(\Bhat(\bhat,\mhat(x))) \simeq 1,
\]
where we also used Lemma~\ref{lem-Bhat-both-dhat}.
Thus all the four factors in~\eqref{eq-prod-rho-rhohat} are $\simeq1$, 
which gives the desired estimate.

On the other hand, if $|x|\le \tfrac12|\bhat|$ then 
Corollary~\ref{cor-d(x,infty)-m(x)} with $y=\bhat$
gives that 
\begin{equation*}  
\dhat(x,\bhat)\simeq 
\frac{1}{\nu(B_x)^{1/\s}} \ge \frac{1}{\nu(B_{\bhat})^{1/\s}}
\simeq 1
\end{equation*} 
and hence
\[
\mhat(x) = \dhat(x,\bhat)+1 \simeq \dhat(x,\bhat) \simeq 
\frac{1}{\nu(B_x)^{1/\s}}.
\]
Thus, the two middle factors in \eqref{eq-prod-rho-rhohat} cancel. 
Then Lemma~\ref{lem-comp-close-balls} 
and Proposition~\ref{prop-nuhat-comp-r} with $r=\mhat(x)$ and $m(x)=|x|$ yield that
\[
\nuhat(\Bhat(\bhat,\mhat(x)))^{1/\s} 
\simeq \nuhat(\Bhat(x,\mhat(x)))^{1/\s}
\simeq (|x|^{-\s})^{1/\s} = |x|^{-1}.
\]
Inserting this into \eqref{eq-prod-rho-rhohat} concludes the proof.
\end{proof}

\begin{thm} \label{thm-spher-flatt}
The sphericalization of the flattened space is equivalent to the original 
space in the following sense:
The metrics $\dtilde$ and $d$ are equivalent and the measures $\nutilde$ and $\nu$ 
are comparable, that is, for all $x,y \in Z$,
\[
 \dtilde (x,y) \simeq d(x,y) \quad \text{and}   \quad d\nutilde (x) \simeq d\nu (x),
\]
where the comparison constants depend only on $C_{\nu}$, $\ka$, $\s$, $R_\infty$, $\nu(Z)$ and $|\bhat|$.
\end{thm}

\begin{proof}
Since $\nutilde(\{\binfty\})=\nu(\{b\})=0$ by Lemma~\ref{lem-measure-growth}, the comparability of the measures follows from Lemma~\ref{lem-rho-rhohat-2-new}.
For the metrics we may assume by symmetry that $|x| \leq |y|$ and $y \neq b$.
Let $a>0$ be a constant, to be chosen later, such that 
$a \le \tfrac12 |\bhat|/R_\infty \le\tfrac12$. 

\textbf{Case 1:} 
$0 < |x| \leq a |y|$.
By Lemma~\ref{lem-measure-growth},
\begin{equation}         \label{eq-nu-x-le-a-nu-y}
\nu(B_x) \leq \Lambda a^{\al} \nu(B_y) \leq \Lambda a^{\al} \nu(Z).
\end{equation} 
Moreover, $|x| \leq a |y| \leq a R_\infty \leq \frac{1}{2} |\bhat|$
and hence by Corollary~\ref{cor-d(x,infty)-m(x)} with $M=2$,
\begin{equation}         \label{eq-22}
\dhat(x,\bhat) \simeq \frac{1}{\nu(B_x)^{1/\s}} \simge a^{- \al / \s},
\end{equation}  
where by \eqref{eq-nu-x-le-a-nu-y} the comparison constants do not depend on $a$.
Since $|y| \le R_\infty \simle |\bhat|$, Lemma~\ref{lem-d(x,y)-le-iff-m(x)},
 \eqref{eq-nu-x-le-a-nu-y}  and~\eqref{eq-22} yield that
\begin{equation}      \label{eq-est-dhat(y,btilde)}
\dhat(y,\bhat) \simle\frac{1}{\nu(B_y)^{1/\s}}
\simle \frac{a^{\al / \s}}{\nu(B_x)^{1/\s}} \simeq a^{\al / \s} \dhat(x,\bhat),
\end{equation}
where the comparison constants do not depend on $a$.
Thus by choosing $a\le \tfrac12 |\bhat|/R_\infty$ 
small enough, 
we see that it follows from \eqref{eq-22} and \eqref{eq-est-dhat(y,btilde)} that
$|x| \leq a |y|$ implies that 
\begin{equation}  \label{eq-mhat-after-a}
\mhat(x) = \dhat(x,\bhat) +1 \geq 2 (\dhat(y,\bhat) + 1) = 2\mhat(y) 
\quad \text{and} \quad 
d(x,y) \simeq |y|.
\end{equation}
For such $x$ and $y$, we then get from Corollary~\ref{cor-d(x,infty)-m(x)} 
and Lemma~\ref{lem-comp-close-balls} that
\begin{equation}      \label{eq-dhathat-nuhat}
\dtilde(x,y) \simeq \frac{1}{\nuhat(\Bhat(\bhat,\mhat(y)))^{1/\s}}
\simeq \frac{1}{\nuhat(\Bhat(y,\mhat(y)))^{1/\s}}.
\end{equation}

If $|y| \geq \frac{1}{2} |\bhat|$, then $\nu(B_y) \simeq 1$ 
and therefore
by \eqref{eq-est-dhat(y,btilde)},
\[
1 \le \mhat(y) = 1 + \dhat(y,\bhat) \simle 1 + \frac{1}{\nu(B_{y})^{1/\s}}
\simeq 1,
\]
and hence by 
the first comparison in~\eqref{eq-dhathat-nuhat} and 
Lemma~\ref{lem-Bhat-both-dhat},
\[
 \dtilde(x,y) \simeq \frac{1}{\nuhat(\Bhat(\bhat,1))^{1/\s}}
 \simeq 1 \simeq |y| \simeq d(x,y).
\]
On the other hand, if $|y| \leq \frac{1}{2} |\bhat|$ then by Corollary~\ref{cor-d(x,infty)-m(x)},
\[
\mhat(y) = 1+ \dhat(y,\bhat) \simeq 1 + \frac{1}{\nu(B_y)^{1/\s}} \simeq \frac{1}{\nu(B_y)^{1/\s}}
\] 
and therefore by \eqref{eq-dhathat-nuhat},  
Proposition~\ref{prop-nuhat-comp-r} with $r=\mhat(y)$ and \eqref{eq-mhat-after-a},
\[
\dtilde(x,y) \simeq \frac{1}{\nuhat(\Bhat(y,\mhat(y)))^{1/\s}} \simeq |y|  \simeq d(x,y).
\]

\textbf{Case 2:} 
$a |y| \le |x| \le |y|$.  
If $|x| \le |y|\leq \frac{1}{2} |\bhat|$, then by Corollary~\ref{cor-d(x,infty)-m(x)},
\[
\dhat(x,\bhat) \simeq \frac{1}{\nu(B_x)^{1/\s}}
 \simeq \frac{1}{\nu(B_y)^{1/\s}} \simeq \dhat(y,\bhat).
\]
On the other hand if $\frac{1}{2} |\bhat| \le |y| \le |x|/a$, 
then we get from 
Lemma~\ref{lem-d(x,y)-le-iff-m(x)} that 
\[
1 \le \mhat(y) = 1 + \dhat(y,\bhat) \simle 1 + \frac{1}{\nu(B_{\bhat})^{1/\s}} \simle 1
\]
and similarly, $\mhat(x) \simeq 1$. 
In both cases, we have $\mhat(x) \simeq \mhat(y)$ and hence by 
Lemma~\ref{lem-d(x,y)-m(x)}\ref{it-y-le-M-x}, applied to both $\dhat$ 
and $d$, together with Lemma~\ref{lem-rho-rhohat-2-new},
\[
\dtilde(x,y) \simeq \rhohat(x) \dhat(x,y) \simeq \rhohat(x) \rho(x)d(x,y) 
\simeq d(x,y).
\]

\textbf{Case 3:}
$x=b$.
Let $\{x_j\}_{j=1}^\infty$ be a sequence in $Z \setminus \{b\}$ such that $d(x_j,b) \to 0$ 
(it exists because $Z$ is uniformly perfect at $b$).
Then $\{x_j\}_{j=1}^\infty$ is a $d$-Cauchy sequence,
and hence by the above Cases~1 and~2 also a $\dtilde$-Cauchy sequence.
Since it cannot $\dtilde$-converge to any $z\in \Zhat = Z\setm\{b\}$, 
it must $\dtilde$-converge to $\binfty$.
Thus we can identify $\binfty$ with $b$.
The comparison $\dtilde(x,y) \simeq d(x,y)$ is then extended to $x=b=\binfty$ 
by continuity.
\end{proof}



\begin{thebibliography}{99}


\bibitem{BaloghBuckley} \art{\auth{Balogh}{Z} \AND \auth{Buckley}{S}}
{Sphericalization and flattening}
{Conform. Geom. Dyn.}{9}{2005}{76--101}

\bibitem{BBbook} \book{\auth{Bj\"orn}{A} \AND \auth{Bj\"orn}{J}}
        {\it Nonlinear Potential Theory on Metric Spaces}
    {EMS Tracts in Mathematics {\bf 17},
        European Math. Soc., Z\"urich, 2011}

\bibitem{BBnoncomp}  \art{\auth{Bj\"orn}{A} \AND \auth{Bj\"orn}{J}}	
        {Poincar\'e inequalities and Newtonian Sobolev functions on noncomplete 
        metric spaces}
        {J. Differential Equations} {266} {2019} {44--69}
	Corrigendum: \emph{ibid.}  {\bf 285} (2021),  493--495.

\bibitem{BBbring} \art{\auth{Bj\"orn}{A} \AND \auth{Bj\"orn}{J}}
        {Sharp Besov capacity estimates for annuli in metric spaces
with doubling measures}
       {Math. Z.} {305} {2023} {Paper No. 41, 26 pp}

\bibitem{BBLi} \art{\auth{Bj\"orn}{A}, \auth{Bj\"orn}{J}  \AND \auth{Li}{X}}
        {Sphericalization and \p-harmonic functions on unbounded
          domains in Ahlfors regular spaces}
        {J. Math. Anal. Appl.}{474}{2019}{852--875}

\bibitem{BBShyptrace} \art{Bj\"orn, A., Bj\"orn, J. \AND Shan\-mu\-ga\-lin\-gam, N.}
  {Extension and trace results for doubling metric measure spaces and their
    hyperbolic fillings}
 {J. Math. Pures Appl.} {159} {2022} {196--249}

\bibitem{BonkKleiner02} \art{\auth{Bonk}{M} \AND \auth{Kleiner}{B}}
        {Rigidity for quasi-M\"obius group actions}
        {J. Differential Geom.} {61}{2002} {81--106} 

\bibitem{BuckleyHerronXie} \art{\auth{Buckley}{S. M}, \auth{Herron}{D}
    \AND \auth{Xie}{X}}
  {Metric space inversions, quasihyperbolic distance, and uniform spaces}
  {Indiana Univ. Math. J.} {57} {2008} {837--890}

\bibitem{ChenKumagai08} \art{\auth{Chen}{Z.-Q} \AND \auth{Kumagai}{T}}
   {Heat kernel estimates for jump processes of mixed types on metric measure spaces}
  {Probability Theory Related Fields} {140} {2008} {277--317}

\bibitem{DL1} \art{\auth{Durand-Cartagena}{E} \AND \auth{Li}{X}}
      {Preservation of \p-Poincar\'e inequality for large $p$ 
     under sphericalization and flattening}
      {Illinois J. Math.} {59} {2015} {1043--1069} 

\bibitem{DL2} \art{\auth{Durand-Cartagena}{E} \AND \auth{Li}{X}}
      {Preservation of bounded geometry under sphericalization and 
    flattening: quasiconvexity and $\infty$-Poincar\'e inequality}
  {Ann. Acad. Sci. Fenn. Math.} {42} {2017} {303--324}

\bibitem{GibaraKorteShan} \art{Gibara, R., Korte, R. \AND Shanmugalingam, N.}
{Solving a Dirichlet problem on unbounded domains via a conformal transformation}
{Math. Ann.}{389}{2024}{2857--2901}

\bibitem{GibaraShan} \art{\auth{Gibara}{R} \AND \auth{Shanmugalingam}{N}}
{Conformal transformation of uniform domains under weights that depend on distance to the boundary}
{Anal. Geom. Metr. Spaces.}{10}{2022}{297--312}


\bibitem{GKS} \art{Gogatishvili, A.,  Koskela, P. \AND Shan\-mu\-ga\-lin\-gam, N.}
 {Interpolation properties of Besov spaces defined on metric spaces}
  {Math. Nachr.}{283}{2010}{215--231}

\bibitem{GKZ} \art{Gogatishvili, A.,  Koskela, P. \AND Zhou, Y.}
  {Characterizations of Besov and Triebel--Lizorkin spaces on metric measure spaces}
  {Forum Math.} {25}{2013} {787--819}

\bibitem{Heinonen} \book{Heinonen, J.}
        {Lectures on Analysis on Metric Spaces}
        {Springer, New York, 2001}

\bibitem{JW84} \book{Jonsson, A. \AND Wallin, H.}
         {Function Spaces on Subsets of\/ $\mathbf{R}^n$}
         {Math. Rep. {\bf 2}:1, Harwood, London, 1984}
         
\bibitem{KajShi} 
 \artprep{\auth{Kajino}{N}   \AND  \auth{Shimizu}{R}} 
             {Contraction properties and differentiability of \p-energy forms with applications 
             to nonlinear potential theory on self-similar sets}
     {\emph{Preprint}, 2024}
\arXiv{2404.13668v2}

         
\bibitem{KajShiToh} 
  \arttoappearin{\auth{Kajino}{N}   \AND  \auth{Shimizu}{R}} 
             {Korevaar--Schoen \p-energy forms and associated energy measures on fractals}
            {\emph{Facets of Contemporary Analysis, Geometry and 
    Non-Euclidean Statistics}, Tohoku Series in Mathematical Sciences {\bf 1},
Springer, Singapore, 2026}        

         
\bibitem{KSS25}  
  \art{\auth{Kumagai}{T}, \auth{Shanmugalingam}{N} \AND \auth{Shimizu}{R}}
             {Finite dimensionality of Besov spaces and potential-theoretic decomposition of metric spaces}
             {Ann. Fenn. Math.}{50}{2025}{347--369}


\bibitem{LiShan} \art{\auth{Li}{X} \AND \auth{Shanmugalingam}{N}}
        {Preservation of bounded geometry under sphericalization and  flattening}
        {Indiana Univ. Math. J.}{64}{2015}{1303--1341}

\bibitem{MalyBesov} \artprep{\auth{Mal\'y}{L}}
  {Trace and extension theorems for Sobolev-type functions in metric spaces}
  {\emph{Preprint}, 2017}
  \arXiv{1704.06344}
  
  
\bibitem{MurShi}   \art{\auth{Murugan}{M} \AND \auth{Shimizu}{R}}
             {First-order Sobolev spaces, self-similar energies and energy measures on the Sierpi\'nski carpet}
             {Comm. Pure Appl. Math.}  {78} {2025} {1523--1608}


\bibitem{KasiaP2010}   \art{\auth{Pietruska-Pa\l uba}{K}}
        {Heat kernel characterisation of Besov--Lipschitz spaces on metric measure spaces}
        {Manuscripta Math.}{131}{2010}{199--214}

\bibitem{SF1990}   \art{\auth{Saloff-Coste}{L}}
        {Analyse sur les groupes de Lie \'a croissance polyn\^omiale}
        {Ark. Mat.}{28}{1990}{315--331}

\bibitem{Shimizu}   \art{\auth{Shimizu}{R}}
        {Construction of \p-energy and associated energy measures on Sierpi\'nski carpets}
        {Trans. Amer. Math. Soc.}{377}{2024}{951--1032}


\bibitem{Wildrick} \art{\auth{Wildrick}{K}}
   {Quasisymmetric parametrizations of two-dimensional metric planes}
   {Proc. Lond. Math. Soc.} {97} {2008} {783--812}

\bibitem{Yang2003}  \art{\auth{Yang}{D}}
        {New characterizations of Haj\l asz--Sobolev spaces on metric spaces}
        {Sci. China Ser.\ A}{46}{2003}{675--689}

\end{thebibliography}
\end{document}